\documentclass{amsart}

\usepackage{amsmath}
\usepackage{amssymb}
\usepackage{amsthm}
\usepackage{array}
\usepackage{subfigure}
\usepackage[all]{xy}
\usepackage{graphics,graphicx}
\usepackage{enumitem}

\usepackage[short,nodayofweek]{datetime}
\usepackage{hyperref}
\usepackage{nicefrac}
\usepackage{verbatim}

\usepackage{tikz}
\usetikzlibrary{matrix,arrows}

\DeclareFontFamily{OT1}{rsfs}{}
\DeclareFontShape{OT1}{rsfs}{n}{it}{<-> rsfs10}{}
\DeclareMathAlphabet{\mathscr}{OT1}{rsfs}{n}{it}

\def\bar{\overline}
\newcommand{\FF}{\mathbb F}
\newcommand{\ZZ}{\mathbb Z}
\newcommand{\RR}{\mathbb R}
\newcommand{\QQ}{\mathbb Q}

\newcommand{\NN}{\mathbb N}

\newcommand{\gs}{\geqslant}
\newcommand{\sm}{\setminus}
\newcommand{\ls}{\leqslant}

\newcommand{\Sec}{\operatorname{Sec}}
\newcommand{\chara}{\operatorname{char}}
\newcommand{\Spec}{\operatorname{Spec}}

\newcommand{\cd}{\operatorname{cd}}

\newcommand{\HH}[3]{\operatorname{H}^{#1}_{#2}(#3)}
\newcommand{\Het}[2]{\operatorname{H}^{#1}_{\mathrm{\acute{e}t}}(#2)}
\newcommand{\Hetl}[3]{\operatorname{H}^{#1}_{\mathrm{\acute{e}t}, #2}(#3)}

\newcommand{\iso}{\cong}

\def\V{\mathcal{V}}
\def\I{\mathcal{I}}
\def\A{\mathfrak{a}}
\def\AA{\mathbb{A}}
\def\B{\mathfrak{b}}
\def\C{\mathfrak{c}}
\def\p{\mathfrak{p}}
\def\L{\Lambda}
\def\Ch{\check{\mcal{C}}}

\def\fm{\mathfrak{m}}
\def\fn{\mathfrak{n}}
\def\mcal{\mathcal}

\newcommand{\kk}{\Bbbk}
\newcommand{\kv}{{\kk[V]}}
\newcommand{\kvv}{{\kk[V^2]}}
\newcommand{\vg}{{V/\! \!/G}}
\newcommand{\kvg}{{\kk[V]^G}}

\newcommand{\svg}{\mathcal{S}_{V,G}}

\newcommand{\diag}{\operatorname{diag}}
\newcommand{\Ll}{\mathcal{L}}
\newcommand{\Ss}{\mathcal{S}}
\newcommand{\supp}{\operatorname{supp}}
\newcommand{\kx}{\kk[x]}
\newcommand{\Nn}{\mathcal{N}}

\newcommand{\proj}{\operatorname{proj}}

\def\PP{\mathbb{P}}

\def\GL{\operatorname{GL}}

\newcommand{\wt}{\mathrm{wt}}
\newcommand{\spn}{\mathrm{span}}
\newcommand{\conv}{\operatorname{conv}}
\newcommand{\convo}{\conv^{\circ}}

\newtheorem{thm}{Theorem}[section]
\newtheorem*{thm*}{Theorem}
\newtheorem{cor}[thm]{Corollary}
\newtheorem*{cor*}{Corollary}
\newtheorem{prop}[thm]{Proposition}
\newtheorem{lem}[thm]{Lemma}

\newtheorem{innercustomthm}{Theorem}
\newenvironment{customthm}[1]
  {\renewcommand\theinnercustomthm{#1}\innercustomthm}
  {\endinnercustomthm}

\theoremstyle{definition}

\newtheorem{ques}[thm]{Question}
\newtheorem{conj}[thm]{Conjecture}

\theoremstyle{remark}
\newtheorem{rmk}[thm]{Remark}
\newtheorem{eg}[thm]{Example}

\newcommand{\done}{\hfill $\triangleleft$}

\title[Mapping toric varieties into low dimensional spaces]{Mapping toric varieties into \\ low dimensional spaces}

\author[Emilie Dufresne]{Emilie Dufresne}
\address{School of Mathematical Sciences, University of Nottingham, University Park, Nottingham, NG7 2RD, UK}
\email{emilie.dufresne@nottingham.ac.uk}

\author[Jack Jeffries]{Jack Jeffries}
\address{University of Michigan, Department of Mathematics, 1825 East Hall, 530 Church Street, Ann Arbor, MI 48109-1043, USA}
\email{jackjeff@umich.edu }

\subjclass[2010]{13A50, 13D45, 14M25}

\keywords{Segre-Veronese varieties, Dimension of secant variety, torus invariants, separating invariants, local cohomology}

\begin{document}

\maketitle

\begin{abstract}
A smooth $d$-dimensional projective variety $X$ can always be embedded into $2d+1$-dimensional space. In contrast, a singular variety may require an arbitrary large ambient space. If we relax our requirement and ask only that the map is injective, then any $d$-dimensional projective variety can be mapped injectively to $2d+1$-dimensional projective space. A natural question then arises: what is the minimal $m$ such that a projective variety can be mapped injectively to $m$-dimensional projective space? In this paper we investigate this question for normal toric varieties, with our most complete results being for Segre-Veronese varieties. 
\end{abstract}

\date{\today}
%------------------------------------------------------------------------
%------------------------------------------------------------------------
%------------------------------------------------------------------------
%------------------------------------------------------------------------

\section{Introduction}

It is well known that a smooth $d$-dimensional projective or affine variety can always be embedded into $\PP^{2d+1}$. This story is different for singular varieties, including affine cones over smooth projective varieties. For example, the affine cone over the $n$th Veronese embedding of $\PP^1$ cannot be embedded in $\AA^m$ for $m<n+1$. In some situations, one may be willing to lose some information and be satisfied with an injective morphism $X \to \PP^m$ or $X\to \AA^m$. 

\begin{ques}\label{q-1}
What is the minimal $m$ such that a projective (affine) variety $X$ can be mapped injectively into projective space (respectively, affine space) of dimension~$m$?
\end{ques}

The bound $2d+1$ holds in general in this setting by the same linear projection argument that yields an embedding of a smooth projective variety into $\PP^{2d+1}$.

Namely, if $X\subseteq \PP^N$ and $p$ is a point in $\PP^N\setminus X$ that does not lie on any secant to $X$, that is on any line that intersects $X$ in at least two points, then the projection from $p$ to a hyperplane will be injective on $X$. One can then repeat this argument until the set of points on secant lines fills the ambient space. The bound $2d+1$ is simply the expected dimension of the secant variety, which is the Zariski closure of the set of points on secant lines. In the affine case, one simply uses the same argument on a projectivization; see ~\cite[Section~5.1]{ed:si} for a similar argument in a special case and \cite[Theorem~5.3]{tk-gk:aitng} for a different argument in a more general context.

Naturally, an absolute lower bound is given by the dimension of $X$. This absolute lower bound is sometimes attained even when $X$ is not itself $\PP^d$ or $\AA^d$, at least in positive characteristic \cite[Example 3.1]{ed:sifrg}. But this is a rare occurence. In characteristic zero, for normal varieties, it does not happen unless $X$ is isomorphic to $\PP^m$ or $\AA^m$, see~\cite[Corollary~4.6]{ed-hk:ismag}. 

In this paper we focus on toric varieties and the affine cones over them. Our most complete results are for Segre-Veronese varieties, a construction simultaneously generalizing Veronese varieties ($r=1$) and Segre varieties (each $a_i=1$). For the affine cone over a Segre-Veronese variety we obtain:
 
 \begin{thm}\label{thm-MinSizeSV}
Let $Y$ be the affine cone over the Segre-Veronese variety $X$ that is the image of the closed embedding $\prod_{i=1}^r \PP^{n_i-1} \hookrightarrow \PP^N$ given by the line bundle $\mcal{O}(a_1,\dots,a_r)$. If $s$ is minimal such that $Y$ can be mapped injectively to $\AA^s$, then
\begin{enumerate}
\item $s=2n_1-1$, if $r=1$ and $a_1$ is not 1 or a power of $\chara \kk$;\label{thm-MinSizeVero}
 \item $s=2(n_1+n_2)-4$, if $r=2$ and $a_1,a_2$ are either 1 or a power of $\chara \kk$;\label{thm-MinSizeSV2}
  \item $s\leq 2\sum_{i=1}^r n_i  - 2r +2$, in general;\label{thm-MinSizeSVupper}
 \item $s\geq 2\sum_{i=1}^r n_i  - 2r +1$, if at least one $a_i$ is not 1 or a power of $\chara \kk$;\label{thm-MinSizeSV1}

 \item $s\geq 2\sum_{i=2}^r n_i  - 2r +4$, if every $a_i$ is either 1 or a power of $\chara \kk$.\label{thm-MinSizeSV3}
\end{enumerate}
\end{thm}

Lower bounds are obtained via techniques better suited to the affine case. If $Y$ is an affine algebraic variety, having an injective morphism $Y\to \AA^m$ corresponds to having a \emph{separating set} $E=\{f_1,\ldots,f_m\}\subseteq \kk[Y]$. The notion of separating set is usually defined relatively to a larger ring of functions $R\supset \kk[Y]$  and in a much more general context, for ring of functions on a set, see \cite[Definition 1.1]{gk:si}. This variant comes up in many different contexts under different names (see for example the introduction of \cite{eds:fdepeei}). 

For the affine cone $Y$ over a Segre-Veronese variety,  we have a surjective map 
$\AA^d\to Y$ where $d={\sum_{i=1}^r n_i}$, dual to the inclusion map of the ring of regular functions on $Y$ into a polynomial ring in $d$ variables. A consequence is that the morphism ${Y\to \AA^m}$ is injective if and only if the natural inclusion of the reduced fiber products 
\[\AA^d \times_Y \AA^d \subseteq \AA^d \times_{\AA^m} \AA^d\]
 is an isomorphism. As $ \AA^d \times_{\AA^m} \AA^{d}$ is the zero set in $ \AA^{2d}$ of 
 \[(f_i\otimes 1-1\otimes f_i\mid i=1,\ldots,m)\]
  it follows that the \emph{arithmetic rank} of the defining ideal of $\AA^d \times_Y \AA^d$, that is the minimal number of generators up to radical, is a lower bound for the minimal $m$ such that $Y$ can be mapped injectively to $\AA^m$ (cf  \cite[Section~3]{ed-jj:silc}). Our arguments exploit the fact that the affine cone over a Segre-Veronese variety is isomorphic to $\vg:=\Spec(\kvg)$, where $V$ is a representation of an algebraic group $G$ and $\kvg:=\{f\in \kv \mid f(u)=f(\sigma\cdot u), \text{ for all }u\in V,~\sigma\in G\}$. Accordingly, most of this paper will be written from that point of view.

If we assume $G$ is reductive (and so the quotient morphism $V\to \vg$ is surjective), then having an injective morphism $\Spec(\kvg)=:\vg\to \AA^m$ corresponds to having a separating set in the sense of  \cite[Section 2.3.2]{hd-gk:cit}, that is, a set $E$ of invariants such that whenever two points of $V$ can be separated by some invariant, they can be separated by an element of $E$.

In this setting, the fibre product $V\times_{\vg}V$ is called the \emph{separating variety} and denoted by $\svg$. The key observation of \cite[Section~3]{ed-jj:silc} is that the minimal cardinality of a separating set is bounded below by the arithmetic rank of the defining ideal of $\svg$. For representations of finite groups the arithmetic rank of the defining ideal of the separating variety ends up being meaningful in terms of the geometry of the representation (see \cite{ed-jj:silc}). As in \cite{ed-jj:silc}, we use the nonvanishing of local cohomology modules to find lower bounds for the arithmetic rank of the defining ideal of the separating variety. For Segre varieties with two factors, this is not conclusive in positive characteristic, and so instead we must use \'etale cohomology. Following \cite{ed-jj:silc}, the general strategy is to decompose the separating variety as a union of simpler objects. The difficulty is, unlike for representations of finite groups, the separating variety is not simply an arrangement of linear subspaces.

Linear projections are often sufficient to reach the minimal $m$ such that a projective variety can be mapped injectively to $\PP^m$. In the case of toric varieties, the image will often no longer be toric. A natural question is the following:

\begin{ques}\label{min-m-mon}
What is the minimal $m$ such that a projective toric variety is mapped injectively to $\PP^m$ so that the image is also a toric variety? Equivalently, what is the minimal cardinality of a monomial separating set for the affine cone over a toric variety?
\end{ques}

In this paper we address this question for normal affine toric varieties. These include the affine cones over Segre-Veronese varieties. The normality assumption ensures that the ring can be identified with the ring of invariants of a representation of an algebraic torus, which provides extra structure. Indeed, the rings of invariants for representations of tori are determined by the combinatorics and convex geometry of the weights. In Proposition~\ref{prop-SVminmon}, we determine the answer to Question~\ref{min-m-mon} for Segre-Veronese varieties; as a consequence, in Corollary~\ref{cor-sparse} we give a bound on the sparsity of a separating set for the associated torus action.

In general, the minimal size of a monomial separating set is much larger than the minimal size of a separating set, but it is often still smaller than the size of a minimal generating set for the ring of invariants. Inspired by \cite{md-es:hdag}, we show the following.

\begin{customthm}{\ref{thm-2r+1}}
  Let $V$ be a $n$ dimensional representation of a torus $T$ of rank $r\ls n$. The invariants involving at most $2r+1$ variables form a separating set.
\end{customthm}

The remainder of the paper is  organized as follows. In Section~\ref{section-setup} we describe the combinatorial set-up and notation we need in order to discuss linear representations of tori, including giving an explicit link between representations of tori and Segre-Veronese varieties. In Section~\ref{section-SepVar} we first give some general results about the decomposition of the separating variety for representation of tori, before giving more explicit results in the case of Segre-Veronese varieties. Section~\ref{section-UpperBounds} focuses on upper bounds on the size of separatating sets and Section~\ref{section-LowerBounds} on the lower bounds. Finally, in Section~\ref{section-MonomialSepSets} we consider monomial separating sets. As well as the results mentioned above, we give a combinatorial characterization of monomial separating sets for representations of tori.

%------------------------------------------------------------------------
%------------------------------------------------------------------------
%------------------------------------------------------------------------
%------------------------------------------------------------------------

\section{Set up and Notation}\label{section-setup}

We work over an algebraically closed field $\kk$ of arbitrary characteristic. The characteristic will sometimes make a difference. We consider a $n$-dimensional representation $V$ of a torus $T$ of rank $r$. Without loss of generality we can assume this is given by the weights $m_1,\ldots,m_n\in \ZZ^r$. That is,  with respect to the basis $\{b_1,\dots ,b_n\}$ of $V$, the action of $T$ on $V$ is given by $t\mapsto \diag(t^{-m_1},\ldots,t^{-m_n})$, where $t$ denotes the element $(t_1,\ldots,t_r)$ of $T$, and $t^{-m_i}$ denotes the element $t_1^{-m_{i,1}}\cdots t_r^{-m_{i,r}}$ of $\kk$. Let $A$ be the matrix whose columns are the $m_i$'s. We will assume throughout that $A$ has full rank $r\ls n$. We will use $\{e_1,\dots,e_r\}$ to denote the standard basis of $\ZZ^r$.

We will write $\kk[x]$ to denote $\kk[x_1,\ldots,x_n]=\kk[V]$, where $\{x_1,\dots,x_n\}$ is the basis  of the 1-forms of $\kk[V]$  dual to the basis $\{b_1,\dots ,b_n\}$ of $V$. In terms of the basis $\{x_1,\dots,x_n\}$, the representation of $T$ takes the form $t\mapsto \diag(t^{m_1},\ldots,t^{m_n})$.

The ring of invariants $\kv^T$ is the monomial algebra given by the semigroup $\Ll:=\ker_{\ZZ}A\cap \NN^n$, that is
\[\kv^T=\spn_\kk\{x^\alpha \mid \alpha\in \Ll\}.\]
The field of rational invariants is similarly given by $\ker_\ZZ A$:
\[\kk(V)^T=\spn_\kk\{x^\beta \mid \beta \in \ker_{\ZZ} A\}.\]
For a natural number $a\in\NN$, we write $[a]$ to denote the set $\{1,\ldots,a\}$. For $I\subseteq [n]$, we set $V_I=\spn_\kk\{b_i \mid i\in I\}$, 
\[\ker_{\ZZ} A_I:=\{\beta\in \ker_{\ZZ} A \mid \beta_j=0 \ \text{for} \  j\notin I\}\,,\]
 and $\Ll_I=\ker_\ZZ A_I \cap \Ll$.
 
 For $u=\sum u_i b_i\in V$, we define $\supp(u):=\{i\in [n]\mid u_i\neq 0\}$ and similarly for $\beta\in \ZZ^n$, ${\supp(\beta):=\{i\in [n]\mid \beta_i\neq 0\}}$. For $I\subseteq \{1,\ldots, n\}$, we define the \emph{weight set} of $I$ to be ${\wt(I):=\{m_i\mid i\in I\}}$, and for $u\in V$ and $\alpha \in \ZZ^n$, we write $\wt(u):=\wt(\supp(u))$ and ${\wt(\alpha):=\wt(\supp(\alpha))}$. Further, we will write $\conv(I)$, respectively $\convo(I)$, to denote the convex hull of $\wt(I)$, respectively the relative interior of the convex hull of $\wt(I)$, which is the interior of $\wt(I)$ with respect to the usual metric topology on the linear span of $\wt(I)\subseteq \RR^n$.
 
%------------------------------------------------------------------------
 
\subsection{Segre-Veronese varieties}\label{section-SV}

We will pay particular attention to the affine cones over Segre-Veronese varieties.

A Segre-Veronese variety is the image of the closed embedding
\[{ {\prod_{i=1}^r \PP^{n_i-1} \hookrightarrow \PP^N}} \qquad \text{for} \ {N=1+\prod_{i=1}^r \binom{n_i+a_i-1}{a_i}}\]
 given by the line bundle $\mcal{O}(a_1,\dots,a_r)$ for some tuple $(a_1,\dots,a_r) \in \NN^r$. Its ring of homogeneous coordinates is
\[S=\kk\big[\ M_1 \cdots M_r \ \big| \ M_i \ \text{is a monomial of degree $a_i$ in the variables} \  x_{i1}, \dots , x_{i n_i} \big]\,.\]
This construction simultaneously generalizes Veronese varieties ($r=1$) and 
Segre varieties (every $a_i=1$).

 If every $a_i$ is coprime to the characteristic of $\kk$, the homogeneous coordinate ring $S$ of the Segre-Veronese is the ring of invariants of the polynomial ring 
 \[R=\kk[x_{i \ell} \ | \ 1 \ls i \ls r, 1 \ls \ell \ls n_i]\,,\]
  under the linear action of a diagonalizable group, given as the product of a torus $T$ of rank $r-1$ acting with weights
  \[m_{i\ell}=\left\{\begin{array}{cl} e_i               & i=1,\ldots,r-1\\
                                              -\sum_{i=1}^{r-1} e_i & i=r
                            \end{array}\right.\,,\]
 and a product of cyclotomic groups $\mu_{a_1} \times \cdots \times \mu_{a_r}=:H$, where the $i$-th factor acts on $x_{i\ell}$ by scalar multiplication. That is, setting $W$ to be the $(\sum_{i=1}^r n_i)$-dimensional $\kk$-vector space dual to the space of 1-forms of $R$, the ring of homogeneous coordinates of the nonmodular Segre-Veronese variety $S$ can be identified with the ring of invariants $R^G=\kk[W]^G$, where $G=T\times H$ acts as described above.

The ring $S$ can also be obtained over any field, up to isomorphism, as the ring of invariants of a representation of a torus. Precisely, set
\[S'=\kk\big[\ x_0 M_1 \cdots M_r \ \big| \ M_i \, \text{is a monomial of degree $a_i$ in the variables} \  x_{i1}, \dots , x_{i n_i} \big]\,,\]
and $R'=\kk[x_0 , x_{i \ell} \ | \ 1 \ls i \ls r, 1 \ls \ell \ls n_i]=\kk[W']$ where $W'$ is the $(1+\sum_{i=1}^r n_i)$-dimensional $\kk$-vector space dual to the space of 1-forms of $R'$. Then $S'$ is the ring of invariants of $R'=\kk[W']$ under the action of a rank $r$ torus with weights $m_0= -\sum_{i=1}^r a_i e_i$ and $m_{i\ell} = e_i$. As both groups are reductive, the isomorphism $S\cong S'$ ensures that finding  separating sets in $S$ and $S'$ is exactly the same, and so any bound established for one holds for the other. This follows from the following key fact: if $G$ is reductive, $E\subseteq \kvg$ is a separating set if and only if the morphism $\vg\to\Spec(\kk[E])$ induced by the inclusion $\kk[E]\subseteq \kvg$ is injective (cf \cite[Theorem~2.2]{ed:sifrg}).

In the following lemma, we set $\operatorname{s}(a_1,\dots,a_r)$ to be the smallest cardinality of a separating set for a representation of a torus with ring of invariants isomorphic to the homogeneous coordinate ring of the Segre-Veronese variety $X$, where $X$ is the image of the closed embedding $\prod_{i=1}^r \PP^{n_i-1} \hookrightarrow \PP^N$ given by the line bundle $\mcal{O}(a_1,\dots,a_r)$.

\begin{lem} \label{lem-nonmodSV}
Let $\kk$ be a field of positive characteristic $p$, and fix $a_1,\dots , a_n$. Write each $a_i=a_i' p^{c_i}$ so that $\operatorname{gcd}(a_i',p)=1$. Then $\operatorname{s}(a_1,\dots,a_r)=\operatorname{s}(a'_1,\dots,a'_r)$.
\end{lem}

\begin{proof}
We will show that $\operatorname{s}(a_1,\dots,a_t)=\operatorname{s}(a_1,\dots,a_{t-1},p a_t)$, and the claim follows.

Let $R_1$ be the ring of functions on the affine cone over the Segre-Veronese variety given by the line bundle $\mcal{O}(a_1,...,a_t)$, $R_2$ be the  ring of functions on the affine cone over the Segre-Veronese variety given by the line bundle $\mcal{O}(a_1,\ldots,a_{t-1},p a_t)$, and $R_3$ be the image of $R_1$ under the map that sends the set of variables ($x_{t,1},\ldots,x_{t,n_t}$) to their $p$th powers. Note that $R_3$ is isomorphic to $R_1$.

Suppose that $\{g_1,\ldots,g_s\}\subset R_1$ form a separating set in $R_1$. Then their images $\{g''_1,\ldots,g''_s\}$ under the map giving the isomorphism $R_1\cong R_3$ form a separating set for $R_2$. Indeed, it will be a separating set for $R_3$ (applying the isomorphism), and the morphism $\Spec(R_2)\to \Spec (R_3)$ induced by the inclusion $R_3\subseteq R_2$ is injective. It follows that an upper bound on the minimal size of separating sets for $R_1$ is also an upper bound for $R_2$.

Since $(R_2)^{p} \subseteq R_3$, taking a separating set for $R_2$ and taking ${p}$th powers produces a separating set for $R_3$ of the same cardinality. Then, applying the isomorphism $R_1\cong R_3$, we get a separating set for $R_1$ of the same size. Therefore, the minimal size of a separating set in $R_1$ is a lower bound for the minimal size of a separating set in $R_2$, completing the proof of the claim.
\end{proof}

%------------------------------------------------------------------------

%------------------------------------------------------------------------

%------------------------------------------------------------------------

\section{The separating variety}\label{section-SepVar}

In this section we describe the separating variety for representations of tori. The \emph{separating variety} $\svg$ is a closed subvariety of $V\times V$ that encodes which points can be separated by invariants. Namely,
\[\Ss_{V,G}:=\{(u,v)\in V\times V \mid f(u)=f(v),~\forall f\in\kvg\}\,.\]
It deserves its name because it characterizes separating sets. Indeed, $E\subseteq \kvg$ is a separating set if and only if 
\[\V_{V\times V}(f\otimes 1-1\otimes f \mid f\in E)=\Ss_{V,G}\] 
(see \cite[Section 2]{gk:cirgpc}). In particular, the defining ideal for the separating variety is the radical of the \emph{separating ideal}
\[ \I_{V,G}:=(f \otimes 1 - 1 \otimes f \mid f\in \kvg)\,.\]
Of course, the separating variety always contains the graph of the action 
\[ \Gamma_{V,G}:=\{(u,\sigma\cdot u)\mid u\in V,\sigma\in G\}\,,\] 
and its Zariski closure $\overline{\Gamma_{V,G}}$.

Let $V$ be a linear representation of an algebraic group $G$. Let $\pi\colon V\to \vg$ be the morphism corresponding to the inclusion $\kvg\subseteq \kv$. The \emph{nullcone} $\Nn_V$ is defined as $\Nn_V:=\pi^{-1}(\pi(0))$. It coincides with the set $\V_V(f\mid f\in \kv^{G}_+)$ of common zeroes in $V$ of all nonconstant homogeneous invariants. Naturally, the product $\Nn_{V,G}\times\Nn_{V,G}$ is always contained in the separating variety. For a representation of a torus, the nullcone can be described in terms of the geometry of the weights of the action:

\begin{lem}[{see for example \cite[Proposition 4.4]{dlw:wirtipr}}] \label{lem-nullcone}
Let $V$ be a representation of a torus $T$. The nullcone is an arrangement of linear subspaces and its decomposition as irreducibles is as follows:
\[\Nn_{V,T}=\bigcup_{\tiny{\begin{array}{c} I \text{ maximal s.\,t.}\\ 0\notin \conv(\wt(I))\end{array}}} V_I.\]
\end{lem}

We obtain a first coarse decomposition of the separating variety:

\begin{prop}\label{prop-SepVarDecompGen}
 Let $V$ be a representation of a torus of rank $r$ and suppose the matrix of weights has rank $r$. Then the separating variety can be written as
\[\Ss_{V,T}=\overline{\Gamma_{V,T}}\ \bigcup\ (\Nn_{V,T}\times \Nn_{V,T})  \ \bigcup_{K,I,J} \ \Gamma_{V_K,T} \oplus (V_{I}\times V_{J})\,,\]
where the second union ranges over all $K$ with $0\in\convo(\wt(K))$ and all $I,J\subseteq \{1,\ldots, n\}\setminus K$ such that $ 0\notin \convo(\wt(K\cup I))$ and $0\notin\convo(\wt(K\cup J))$.
\end{prop}
\begin{proof}
We first show the inclusion ``$\supseteq$''. By the discussion above, $\overline{\Gamma_{V,T}}$ and $\Nn_{V,T}\times \Nn_{V,T}$ are both contained in the separating variety. It remains to show that for each choice of $K,I,J$, the set $\Gamma_{V_K,T} \oplus (V_{I}\times V_{J})$ is contained in the separating variety. Let $(u_1,u_2)$ be an arbitrary point of $\Gamma_{V_K,T} \oplus (V_{I}\times V_{J})$. By definition we can write $(u_1,u_2)=(v_1,v_2)+(w_1,w_2)$, where $(v_1,v_2)\in \Gamma_{V_K,T}$ and $(w_1,w_2)\in V_{I}\times V_{J}$. Take ${\alpha\in\ker_{\ZZ}A\cap \NN^n}$. As $0$ is not in the interior of the convex hull of $\wt(K\cup I)$ and $\wt(K\cup J)$, it follows that either $\supp(\alpha)\subseteq K$ or $\supp(\alpha)\not\subseteq K\cup I$ and $\supp(\alpha)\not\subseteq K\cup J$. If $\supp(\alpha)\subseteq K$, then
\[x^\alpha(u_1)=x^\alpha(v_1)=x^\alpha(v_2)=x^\alpha(u_2)\,,\]
and if $\supp(\alpha)\not\subseteq K\cup I$ and $\supp(\alpha)\not\subseteq K\cup J$, then
\[x^\alpha(u_1)=0=x^\alpha(u_2)\,.\]
In both cases $(u_1,u_2)\in \Ss_{V,T}$ as desired.

We now prove the reverse inclusion ``$\subseteq$''. Take $(u_1,u_2)\in \Ss_{V,T}$. As $T$ is a reductive group, this is equivalent to $\overline{Tu_1}\cap \overline{Tu_2}\neq \emptyset$ (follows from \cite[Corollary 3.5.2]{pen:impos}). Without loss of generality we have $z\in \overline{Tu_1}\cap \overline{Tu_2}$, where $Tz=\overline{Tz}$ is the unique closed orbit in $\overline{Tu_1}$ and $\overline{Tu_2}$. If $Tu_1=Tz$, we have
\[(u_1,u_2)\in \overline{Tu_2}\times \{u_2\}\subseteq \overline{\Gamma_{V,T}}\,.\]
Similarly, if $Tu_2=Tz$, then $(u_1,u_2)\in\overline{\Gamma_{V,T}}$. We now suppose that $Tz\neq Tu_1,Tu_2$. If $z=0$, then $(u,v)\in \Nn_{V,T}\times \Nn_{V,T}$, so we suppose $z\neq 0$. Then, since the orbit of $z$ is closed, $0$ is in the interior of the convex hull of $\wt(z)$ (see Lemma \ref{lem-closedorbit} below). By the extended Hilbert-Mumford criterion \cite[Theorem~C]{rwr:oooagalg}, there exist 1-dimensional subtori $S_1,S_2\subseteq T$ such that the intersections $\overline{S_1u_1}\cap Tz$ and $\overline{S_2u_2} \cap Tz$ are nonempty; that is, there exist $t_1,t_2\in T$ such that $t_1\cdot z\in \overline{S_1u_1}$ and $t_2\cdot z\in \overline{S_2u_2}$. If $t_1\cdot z\in S_1u_1$, then $Tu_1=Tz$, and this case is done. Similarly, $t_2\cdot z\in S_2u_2$ is also done. So we now suppose that $t_1\cdot z\notin S_1u_1$ and $t_2\cdot z\notin S_2u_2$. The 1-dimensional subtori $S_1,S_2$ correspond to a choice of $\delta_1,\delta_2\in \ZZ^r$. We then have
\[S_1u_1=\{(s^{\delta_1\cdot m_1}u_{1,1},\ldots, s^{\delta_1\cdot m_n}u_{1,n}) \mid s\in \kk^*)\}\,,\]
and
\[S_2u_2=\{(s^{\delta_2\cdot m_1}u_{2,1},\ldots, s^{\delta_2\cdot m_n}u_{2,n}) \mid s\in \kk^*)\}\,.\]
Without loss of generality, our assumption that $t_1\cdot z\notin S_1u_1$ and $t_2\cdot z\notin S_2u_2$ implies that $t_1\cdot z$ and $t_2\cdot z$ are seen to belong to the orbit closures by letting  $s$ tends to zero in the above. It then follows that
\begin{align*}
   \delta_1\cdot m_i >0&, \forall i\in \supp(u_1)\setminus \supp(z)\,,\\
   \delta_2\cdot m_i >0&, \forall i\in \supp(u_2)\setminus \supp(z)\,,\\
   \delta_1\cdot m_i=\delta_2\cdot m_i=0&, \forall i\in \supp(z)\,,
  \end{align*}
and so we have $u_1=t_1\cdot z+w_1$ and $u_2=t_2\cdot z+w_2$, where $w_1,w_2\in\Nn_{V,T}$ and the intersection of their support with the support of $z$ is empty. Our assumptions that $Tz$ is the unique closed orbit in $\overline{Tu_1}$ and $\overline{Tu_2}$ and it is not equal to $Tu_1$ or $Tu_2$ implies that the orbits $Tu_1$ and $Tu_2$ are not closed. As a consequence, $0$ is not in the interior of the convex hull of $\wt(u_1)$ or $\wt(u_2)$. Writing 
\[v_1:=t_1\cdot z, \ v_2:=t_2\cdot z, \ \text{and} \ K:=\supp z,\  I:=\supp(u_1)\setminus K, \ J:=\supp(u_2)\setminus K\,, \]
 we have $0\in\convo(\wt(K))$, $0\notin \convo(\wt(K\cup I)),\convo(\wt(K\cup J))$, and
$(u_1,u_2)=(v_1,v_2)+(w_1,w_2)$ with $(v_1,v_2)\in\Gamma_{V_K,T}$ and $(w_1,w_2)\in V_I\times V_J$. This completes the proof.
\end{proof}

The following is stated without proof in the characteristic zero case in \cite[6.15]{irs:egiv}. We expect that it is already known, but include a proof for lack of an appropriate reference.
\begin{lem}[{cf. \cite[6.15]{irs:egiv}}]\label{lem-closedorbit}
 The orbit of $z$ is closed if and only if $0\in\convo(\wt(z))$.
\end{lem}
\begin{proof}
By the extended Hilbert-Mumford criterion \cite[Theorem~C]{rwr:oooagalg}, it suffices to verify that any for any one parameter subgroup $\lambda$ of $T$, the limit of $\lambda z$ is contained in $T z$. Thus, $Tz$ is closed if and only if there does not exist $\delta \in \ZZ^r$ such that for all $i \in \supp \delta$, we have $ \delta \cdot m_i >0$. By the hyperplane separation theorem applied to the convex sets $U = \mathrm{int}(\mathrm{conv}\{ m_i \} )$ and $V=\{0\}$, it follows that a real vector $\delta$ satisfying the above exists if and only if $0 \in \convo(\wt(z))$. If a real such $\delta$ exists, it may be perturbed slightly in any direction except one parallel to a subspace generated by a set of $m_i$ that contain $0$ in its convex hull. However, the coordinates in such a subspace are necessarily rational.
\end{proof}

\begin{cor} \label{cor-sepvar1approx}
 Let $V$ be a representation of a torus of rank $r$ and suppose the matrix of weights has rank $r$. Suppose that $0\in\convo(\wt(I))$ implies that $\spn_\RR \wt(I)=\RR^r$. Then the separating variety can be written as 
 \[\Ss_{V,T}=\overline{\Gamma_{V,T}}\ \bigcup \ \Nn_{V,T}\times \Nn_{V,T}.\]
\end{cor}
\begin{proof}
 Suppose $0\in\convo(\wt(K))$, where $K$ is as in Proposition~\ref{prop-SepVarDecompGen}. The assumption that $\wt(K)$ spans $\RR^r$ implies that $0$ is in the interior of the convex hull of $\wt(I)$ for any set containing $K$. Hence the third possible contribution in the statement of Proposition \ref{prop-SepVarDecompGen} does not occur.
\end{proof}

The following lemma gives a step towards establishing which irreducible components of $ \Nn_{V,T}\times \Nn_{V,T}$ are contained in $\overline{\Gamma_{V,T}}$.

\begin{lem}\label{lem-NullconeGraph}
Let $V_I\times V_J$ be an irreducible component of $\Nn_{V,T}\times \Nn_{V,T}$.
\begin{enumerate}
\item If $I\cap J=\emptyset$, then $V_I\times V_J \in \overline{\Gamma_{V,T}}$. \label{lem-uvinGraph}
\item Let $(u,v)\in V_I\times V_J$ have full support. If $\ker_{\ZZ} A_{I\cap J}\neq \{0\}$, then $(u,v)\notin \overline{\Gamma_{V,T}}$. \label{lem-uvnotinGraph}
\end{enumerate}
\end{lem}
\begin{proof}
(1): As $I$ and $J$ are disjoint, the maximality of $I$ implies that for each $j\in J$ there is an invariant with exponent vector $\alpha\in\NN^n$ with $j\in\supp(\alpha)\subseteq I\cup \{j\}$. Hence $\alpha_jm_j+\sum_{i\in I}\alpha_im_i=0$. Using the Hilbert-Mumford criteria \cite[Theorem~C]{rwr:oooagalg}, one can see that supposing $V_I$ is a component of the null cone implies that $I$ is maximal among subsets $K$ of $\{1,\ldots,n\}$ such that there exist $\delta\in \ZZ^r$ satisfying $\delta\cdot m_i> 0$ for all $i\in I$. It follows that
\[\delta\cdot m_j=-\nicefrac{1}{\alpha_j}\sum_{i\in I} \alpha_i(\delta\cdot m_i)<0.\]
The $r$-tuple $\delta$ corresponds to a 1-parameter subgroup of $T$, where the induced action of $t\in\kk^*$ and on a vector $w\in V$ is given by $t\cdot w:=(t^{\delta\cdot m_1}w_1,\ldots,t^{\delta\cdot m_n}w_n)$. For each $t\in\kk^*$ and $(u,v)\in V_I \times V_J$,
\begin{equation} 
(u+t^{-1}\cdot v, t\cdot (u+t^{-1}\cdot v))=(u+t^{-1}\cdot v,t\cdot u +v)  \label{eqn-sequence}
\end{equation}
 belongs to the graph. Note that $\delta\cdot m_i> 0$ for $i\in I$ and $\delta\cdot m_j< 0$ for $j\in J$ imply that (\ref{eqn-sequence}) is also well defined for $t=0$. It follows that $(u+t^{-1}\cdot v,t\cdot u +v)|_{t=0}=(u,v)$ must belong to the Zariski closure of the graph $\bar{\Gamma_{V,T}}$.

(2): If $\ker_{\ZZ} A_{I\cap J}\neq \{0\}$, then there is a rational invariant with support contained in $\supp(u) \cap \supp(v)$, with exponent vector $\beta\in \ZZ^n$. Define $\beta^+\in \ZZ^n$ by ${(\beta^+)_i=\max\{ \beta_i,0\}}$, $\beta^-\in \ZZ^n$ by $(\beta^-)_i=\max\{ -\beta_i,0\}$, and fix $i_0\in \supp(\beta)$. Without loss of generality we may assume that $i_0\in\supp(\beta^+)$. Note that since $A$ does not have zero columns, $\beta\in \ker_{\ZZ} A_{I\cap J}$ implies that $|\supp(\beta)|\gs 2$. Define

\begin{align*}
    u_i &=  \left\{\begin{array}{cl}
		1 & i\in \supp (\beta)\\
		0 & \text{otherwise}\end{array}\right.\\
    v_i &=  \left\{\begin{array}{cl}
		1 & i\in \supp (\beta)\setminus \{i_0\}\\
		0 & \text{otherwise}\end{array}\right. .
   \end{align*}
Then $(u,v)\in V_I\times V_J$ and
\[(x^{\beta^+}\otimes x^{\beta^-}-x^{\beta^-}\otimes x^{\beta^+})(u,v)=x^{\beta^+}(u)x^{\beta^-}(v)-x^{\beta^-}(u)x^{\beta^+}(v)=1\neq 0.\]
That is $(u,v)\notin \bar{\Gamma_{V,T}}$ and so $V_I\times V_J\not\subseteq \bar{\Gamma_{V,T}}$.
\end{proof}

%------------------------------------------------------------------------
\subsection{The separating variety for the affine cone over Segre-Veronese varieties}

In this subsection we consider the case of nonmodular Segre-Veronese varieties, that is such that $a_i\in\kk^*$ for each $i=1,\ldots,r$, which suffices to determine the minimal size of separating sets in general by Lemma \ref{lem-nonmodSV} .

\begin{prop}[Nonmodular Segre-Veronese]\label{prop-sepvarSV}
Consider the nonmodular  Segre-Veronese variety that is the image of the closed embedding $\prod_{i=1}^r \PP^{n_i-1} \hookrightarrow \PP^N$ given by the line bundle $\mcal{O}(a_1,\dots,a_r)$ and whose ring of homogeneous coordinates is identified with the ring of invariants $\kk[W]^G$ as described in Section~\ref{section-SV}.
\begin{enumerate}
\item The separating variety $\Ss_{W,G}$ decomposes as follows: \label{prop-sepvarSV1}
 \[\Ss_{W,G}=\bigcup_{\sigma\in H} (1,\sigma)(\Ss_{W,T})\,.\]
\item For $\sigma,\tau\in H$, if $(1,\sigma)(\Ss_{W,T})$ and $(1,\tau)(\Ss_{W,T})$ are distinct, then their intersection is $\Nn_{W,T}\times\Nn_{W,T}$. \label{prop-sepvarSVinter}
\item For $r=1$, $\Ss_{W,T}=W\times W$. \label{prop-sepvarS1}
\item For $r=2$, the decomposition of the separating variety $\Ss_{W,T}$ as a union of irreducibles is  \label{prop-sepvarS2}
\[\Ss_{W,T}=\overline{\Gamma_{W,T}}\cup W_{\hat{1}}\times W_{\hat{1}} \cup W_{\hat{2}}\times W_{\hat{2}},\]
where $W_{\hat{k}}=\spn\{w_{i,j}\mid i\neq k\}$ in terms of the diagonal basis $\{w_{1,*},w_{2,*}\}$ of $W$. Furthermore, $\overline{\Gamma_{W,T}}$ is cut out by the ideal of $2\times 2$ minors
\[I_2\left(\begin{array}{cccccc}
x_{1,1}\otimes 1 & \cdots & x_{1,n_1}\otimes 1 & 1\otimes x_{2,1} & \cdots & 1\otimes x_{2,n_2} \\
1\otimes x_{1,1} & \cdots & 1\otimes x_{1,n_1} & x_{2,1}\otimes 1 & \cdots & x_{2,n_2}\otimes 1
\end{array}\right).\]

\item For $r\gs 3$, the decomposition of the separating variety as a union of irreducibles is\label{prop-sepvarS3}
\[\Ss_{W,T}=\overline{\Gamma_{W,T}} \cup \bigcup_{k,\ell=1}^r W_{\hat{k}}\times W_{\hat{\ell}}.\]
\end{enumerate}
\end{prop}
\begin{proof}
%1
 As $G$ is abelian, $T\trianglelefteq G$ with $G/T\cong H$ and $\kk[W]^G=(\kk[W]^T)^H$. That is, the quotient $\pi_G\colon W\to W/\!\!/G$ factors through the quotients $\pi_T\colon W\to W/\!\!/T$ and $\pi_H^T\colon  W/\!\!/T\to (W/\!\!/T)/\!\!/H\cong W/\!\!/G$. By definition of the separating variety, we have
 \[\begin{aligned}
    \Ss_{W,G} & =\{(u,v)\in W\times W \mid \pi_G(u)=\pi_G(v)\}\\
              & =\{(u,v)\in W\times W \mid \pi_H^T(\pi_T(u))=\pi_H^T(\pi_T(v))\}\,.
   \end{aligned}\]
As $H$ is a finite group, it follows that   
\[\begin{aligned}
 \Ss_{W,G}  & =\{(u,v)\in W\times W \mid \exists \sigma\in H,~\pi_T(u)=\sigma\cdot \pi_T(v)=\pi_T(\sigma\cdot v)\}\\
                     & =\{(u,v)\in W\times W \mid \exists \sigma\in H,~(u,\sigma\cdot v)\in\Ss_{W,T}\}\\
                     & =\{(u,v)\in W\times W \mid \exists \sigma\in H,~(u,v)\in (1,\sigma^{-1}(\Ss_{W,T})\}\\
                     &=\bigcup_{\sigma\in H} (1,\sigma)(\Ss_{W,T})\,.
  \end{aligned}\]

%2
Lemma \ref{lem-equal} implies that $(1,\sigma)(\Ss_{W,T})=(1,\tau)(\Ss_{W,T})$ if and only if $\sigma^{-1}\tau$ acts trivially on $W/\!\!/T$, since 
\[ {(1,\sigma)(\Ss_{W,T})=(1,\tau)(\Ss_{W,T})} \quad \text{ if and only if } \quad  {(1,1)(\Ss_{W,T})=(1,\sigma^{-1}\tau)(\Ss_{W,T})}\,.\]
So, if $(1,\sigma)(\Ss_{W,T})$ and $(1,\tau)(\Ss_{W,T})$ are distinct, then $\sigma^{-1}\tau$ acts nontrivially on $W/\!\!/T$. The $H$-action on $W/\!\!/T$ extends naturally to the representation ${\rho\colon H \to \GL(V)}$, where $V$ has basis $\{e_{j_1,\ldots,j_r} \mid j_i\in [n_i]\}$, $\pi_T(0)$ coincides with the origin in $V$ and $\rho(\gamma)(e_{j_1,\ldots,j_r})=\prod_{i=1}^{r}\zeta_{a_i}^{m_i}e_{j_1,\ldots,j_r}$ for each $\gamma=(\zeta_{a_1}^{m_1},\ldots,\zeta_{a_r}^{m_r})\in H$. Hence $\sigma^{-1}\tau$ acts nontrivially on $W/\!\!/T$ if and only if it acts nontrivially on $V$, but then $V^{\sigma^{-1}\tau}$ is simply the origin. It then follows that $(W/\!\!/T)^{\sigma^{-1}\tau}=\pi_T(0)$.

An element of this intersection $(1,\sigma)(\Ss_{W,T})\cap(1,\tau)(\Ss_{W,T})$ will be of the form $(u,\sigma\cdot v)=(u',\tau\cdot v')$ for some $(u,v),(u',v')\in \Ss_{W,T}$. We will have $u=u'$ and $v=\sigma^{-1}\tau\cdot v'$, and so
\[\pi_T(v')=\pi_T(u')=\pi_T(u)=\pi_T(v)=\pi_T(\sigma^{-1}\tau\cdot v')=\sigma^{-1}\tau\cdot \pi_T(v')\,;\]
that is, 
\[\pi_T(u)=\pi_T(v)=\pi_T(v')\in (W/\!\!/T)^{\sigma^{-1}\tau}=\{\pi_T(0)\}\,.\]
 As 
 \[\pi_T(\sigma\cdot v)=\sigma\cdot\pi_T(v)=\sigma\cdot\pi_T(0)=\pi_T(0)\,,\] we conclude that $(1,\sigma)(\Ss_{W,T})\cap(1,\tau)(\Ss_{W,T})=\Nn_{W,T}\times\Nn_{W,T}$ as desired.

 %3
 Statement \ref{prop-sepvarS1} is clear since there are no nonconstant invariants. So suppose $r\gs 2$.  Observe that $0$ is not in any proper subset of the weights. In particular, the conditions of Corollary \ref{cor-sepvar1approx} are met and so
  \[\Ss_{W,T}=\overline{\Gamma_{W,T}}\ \bigcup \ \Nn_{W,T}\times \Nn_{W,T}\,.\]
 Furthermore, by Lemma \ref{lem-nullcone}, we have
 \[\Nn_{W,T}=\bigcup_{k=1}^r W_{\hat{k}}.\]
 It remains to establish which products $W_{\hat{k}}\times W_{\hat{\ell}}$ are in $\overline{\Gamma_{W,T}}$. 
 %4
 We first consider the case $r=2$. In this case the intersection of the two components is the origin. Hence, by Lemma~\ref{lem-NullconeGraph}, the decomposition of the separating variety is as in Statement~(\ref{prop-sepvarS2}). The ideal given in Statement~(\ref{prop-sepvarS2}) is exactly the toric ideal of the Lawrence lifting of the matrix of weights $A$, 
 \[\Lambda(A):=\left(\begin{array}{cc} A & 0\\ I & I\end{array}\right),\]
 where $I$ denotes the $(N+1)\times (N+1)$ identity matrix. Thus it is prime and has height $N$ (see \cite[Chapter 7, page 55]{bs:gbcp}). It is easy to see that this ideal vanishes on the graph, and so its zero set in $W\times W$ contains the closure of the graph. As $T$ is connected and $A$ has rank $1$, the closure of the graph is itself an irreducible variety of dimension $N+2$. It follows that the toric ideal associated to $\Lambda(A)$ is the defining ideal of $\overline{\Gamma_{W,T}}$.
 
 %5
 Let us now consider the case $r\gs 3$. First note that in this case, the intersection of two irreducible components of the nullcone will contain the weight space of at least one weight. Let $K$ be the support of this intersection. As the weight space of each weight has dimension at least 2, it follows that $\ker A_K\neq \{0\}$, and so by Lemma~\ref{lem-NullconeGraph}, the product $W_{\hat{k}}\times W_{\hat{\ell}}$ is never contained in $\overline{\Gamma_{W,T}}$, and the decomposition is as stated.

\end{proof}

\begin{lem}
We have $\Nn_{W,T}=\bigcup_{k=1}^r W_{\hat{k}}$, where $W_{\hat{k}}=\spn\{w_{i,j}\mid i\neq k\}$.
\end{lem}
\begin{proof}
 Follows directly from Lemma \ref{lem-nullcone} since in this case zero is not in the convex hull of any proper subset of $wt([n])$. Indeed, any proper subset of the $r+1$ distinct weights are linearly independent.
\end{proof}

\begin{lem}\label{lem-equal}
$(1,1)(\Ss_{W,T})=(1,\sigma)(\Ss_{W,T})$ if and only if $\sigma$ acts trivially on $W/\!\!/T$.
\end{lem}
\begin{proof}
Suppose $(1,1)(\Ss_{W,T})=(1,\sigma)(\Ss_{W,T})$. Then for any $(u,v)\in \Ss_{W,T}$ there exists $(u',v')\in \Ss_{W,T}$ such that $(u,v)=(u',\sigma\cdot v')$. It follows that
\[\pi_T(u)=\pi_T(v)=\pi_T(\sigma\cdot v')=\sigma\cdot\pi_T(v')=\sigma\cdot\pi_T(u')=\sigma\cdot \pi_T(u)\,,\]
and so $\pi_T(u)=\sigma \pi_T(u)$. As we can choose $u$ arbitrarily and $\pi_T$ is surjective (since $T$ is a reductive group, see for example \cite[Lemma 2.3.1]{hd-gk:cit}), it follows that $\sigma$ acts trivially on $W/\!\!/T$.

On the other hand, suppose $\sigma$ acts trivially on $W/\!\!/T$, then of course so does $\sigma^{-1}$. For any $(u,v)\in\Ss_{W,T}$, we will have $(u,v)=(u, \sigma\cdot(\sigma^{-1}\cdot v))\in (1,\sigma)(\Ss_{W,T})$, since
\[\pi_T(u)=\pi_T(v)=\sigma^{-1}\cdot\pi_T(v)=\pi_T(\sigma^{-1}\cdot v)\,.\]
\end{proof}

%------------------------------------------------------------------------

%-----------------------------------------------------------------------

%-----------------------------------------------------------------------

\section{Upper bounds on the size of separating sets}\label{section-UpperBounds} 

One can find an upper bound on the size of separating sets, given some knowledge of the secant variety of the embedding. In this section, for a projective variety $X \subseteq \PP^n$, we define the \emph{secant set} of $X$ to be

\[ \sigma (X) = \bigcup_{ x,x' \in X, x \not= x'} \langle x , x' \rangle \subseteq \PP^n,\]
where $\langle \ \rangle$ denotes linear span. The \emph{secant variety} of $X$ is the closure ${\Sec(X)}=\overline{\sigma(X)}$ of the secant set of $X$.

If $p \in \PP^n$, we write $\pi_p$ for projection from $p$ onto a hyperplane.

\begin{lem}\label{lem-spec-iso-projection} Let $A$ be a graded $\kk$-algebra generated in a single degree, and let $X=\proj A$. If $p\notin \sigma(X)$, then $\pi_p$ induces a bijective map of $\Spec A$ onto its image.
\end{lem}
\begin{proof} We choose coordinates so that $p=[0:\cdots:0:1]$. We lift the map $\pi_p$ to $\pi_{p,\text{aff}}:\mathbb{A}^{n+1} \to \mathbb{A}^{n}$ as $(x_1 ,\dots ,x_n , x_{n+1}) \mapsto (x_1 , \dots ,x_n)$. Now, the fiber over $\pi_{p,\text{aff}}|_{\Spec A}^{-1}(0)$ is just 0 since otherwise $[ 0:\cdots :0:1]$) would have to be in $X$, a contradiction. Suppose there are two points in $\Spec A$ that are mapped to the same point. The only possibility for this is if they are of the form $(a_1 ,\dots, a_n , b)$ and $(a_1 ,\dots, a_n , b') \in \Spec A$. But this forces $[0: \cdots : 0 :1]\in\sigma(X)$, again a contradiction.
\end{proof}

\begin{cor}\label{cor-upperbound}
Let $R$ be a subalgebra of a standard graded polynomial ring. Suppose that there is a separating set for $R$ consisting of homogeneous polynomials of the same degree and let $A\subseteq R$ be the subalgebra it generates. Let $X=\proj A$. Any set of $\dim \Sec(X)+1$ generic linear combinations of the original separating set will be a separating set. 
 \end{cor}
 \begin{proof} Without loss of generality, we may assume that $R=A$, and that the given separating set consists of $s$ linearly independent elements $\{f_1,\dots,f_s\}$. There is a surjection $\kk[y_1,\dots,y_s] \twoheadrightarrow A$ given  by sending $y_i$ to $f_i$; this is degree-preserving if each $y_i$ is assigned the (same) degree of each $f_i$. This map corresponds to the inclusion of $X$ into a projective space $\PP^{s-1}$.
 
 Suppose that there exists a point $p\notin \sigma(X)$. Then $\pi_p:\PP^{s-1} \twoheadrightarrow \PP^{s-2}$ descends to a map from $X$ to its image $\pi_p(X)\subseteq \PP^{s-2}$. By Lemma~\ref{lem-spec-iso-projection}, an affine lift of this map to the homogeneous coordinate rings is a bijection. Since $\pi_p(X)\subseteq \PP^{s-2}$, its homogeneous coordinate ring $A'$ is generated by $s-1$ homogeneous elements (that, by definition of projection, are linear combinations of the $f_i$'s), and since $A'\hookrightarrow A$ is bijective on Spec, the generators of $A'$ are a separating set for $A$.
 
 Now, suppose that $s>\dim \Sec(X)+1$. We claim that $s-1$ generic linear combinations of $\{f_1,\dots,f_s\}$ form a separating set for $A$, and that, if they generate the algebra $A'$, $\dim\Sec(\proj(A')) \ls \Sec(X)$. Indeed, in this case a generic point of $\PP^{s-1}$ lies outside of $\Sec(X)$. By the paragraph above, it follows that projection from a generic point of $\PP^{s-1}$ yields a separating set of size $s-1$ as the generators of the homogeneous coordinate ring of the image; the generators of a (generic) projection are simply a (generic) linear combination of the $f_i$'s. Then, since projection preserves linear incidence, $\Sec(X)$ surjects onto $\Sec(\pi_p(X))$.
 
 The first claim of the corollary follows, since, given $\{f_1,\dots,f_s\}$, one may repeatedly pick a generic point and project until the cardinality of the separating set is no larger than the dimension of the secant variety of $X$.
 \end{proof}
 
 \begin{rmk} The corollary above does not require $R$ to be the invariant ring of a representation of a torus.
 \end{rmk}

\begin{rmk} The statement of Corollary~\ref{cor-upperbound} may also be justified as follows. Recall that the \emph{analytic spread}  of an ideal $I$, $\ell(I)$, in a graded ring $(R,\fm,\kk)$ is the smallest size of a generating set for a minimal reduction of an ideal; if the residue field of the ring is infinite then $\ell(I)$ generic linear combinations of the minimal generators generates a minimal reduction. This number also coincides with the dimension of the special fiber ring $R[It]\otimes \kk$. If $I$ is generated in a single degree $d$, the special fiber ring is a subalgebra of $R$ generated by minimal generators of $I$. See \cite[Chapter~5]{SwansonHuneke} for a thorough treatment of analytic spread.

We claim that the special fiber ring of $\I_{V,G}$ is the coordinate ring of $\Sec(\proj(R^G))$. Indeed, this secant variety is the projectivization of the set of points of the form 
\begin{align*}
&\big(\,a f_1(v) + a' f_1(v')\,,\,\dots\,,\,a f_t(v) + a' f_t(v')\,\big) \\
&=\big(\, f_1(v/\sqrt[d]{a}) - f_1(v'/\sqrt[d]{-a'})\,,\,\dots\,,\,f_t(v/\sqrt[d]{a}) - f_t(v'/\sqrt[d]{-a'})\,\big)
\end{align*}
 where $a,a'\in \kk, v,v' \in V,$ and $f_1,\dots,f_t$ are minimal generators for $R^G$, and hence its coordinate ring is isomorphic to $\kk[f_1\otimes 1 - 1 \otimes f_1, \dots, f_t\otimes 1 - 1 \otimes f_t]$. 
 
 Consequently, the analytic spread of $\I_{V,G}$ is $s=\dim \Sec(\proj(R^G)) +1$. For a generic $s \times t$ matrix of scalars $A$, we have that $[f_1\otimes 1 - 1 \otimes f_1, \dots, f_t\otimes 1 - 1 \otimes f_t] \cdot A$ generates a minimal reduction $J$ of $\I_{V,G}$, and hence agrees with $\I_{V,G}$ up to radical. But then, setting $[f_1, \dots, f_t] \cdot A = [g_1, \dots, g_s]$, we have that 
\[J=(g_1\otimes 1 - 1 \otimes g_1, \dots, g_s\otimes 1 - 1 \otimes g_s)\,.\]
 Thus, $(g_1,\dots,g_s)$ is a separating set for $G$.
\end{rmk}

\begin{eg}[Veronese varieties]\label{eg-veroneses-example}
We consider a Veronese variety that is the image of the closed embedding $ \PP^{n_1-1} \hookrightarrow \PP^N$ given by the line bundle $\mcal{O}(a_1)$ and we suppose that $a_1$ is not 1 or a power of $\chara \kk$. By Lemma \ref{lem-nonmodSV}, it is enough to consider the nonmodular case. As discussed in Section \ref{section-SV}, its ring of homogeneous coordinates is equal to the ring of invariants of the cyclotomic group $\mu_{a_1}$ acting diagonally. The secant variety of this Veronese variety has dimension $2(n-1)$ when $d=2$ and $2(n-1)+1$ otherwise (classical). Hence Corollary \ref{cor-upperbound} implies that the minimal size of a separating set for the affine cone is at most $2n-1$ when $d=2$ and $2n$, otherwise. On the other hand, for all $d$, one can construct a separating set of size $2n-1$ (see \cite[Proposition 5.2.2]{ed:si}) and this is the minimal size of a separating set (follows from \cite[Theorem 3.4]{ed-jj:silc}). Note that the invariants forming this separating set are linear combinations of monomials from the minimal generating set given above, that is, they come from a (nongeneric) linear projection of the Veronese variety.\done
\end{eg}

\begin{prop}\label{prop-UBSV}[Upper bounds on the size of separating sets for the affine cone over Segre-Veroneses]
We consider the Segre-Veronese variety which is the image of the closed embedding $\prod_{i=1}^r \PP^{n_i-1} \hookrightarrow \PP^N$ given by the line bundle $\mcal{O}(a_1,\dots,a_r)$. Then the minimal size of a separating set for the affine cone is bounded above by
\begin{enumerate}
 \item $2n_1-1$, if $r=1$;\label{prop-secantBoundVero}
 \item $2(n_1+n_2)-4$, if $r=2$ and $a_1,a_2$ are either 1 or a power of $\chara\kk$.\label{prop-secantBound2}
 \item $2\sum_{i=1}^r n_i  - 2r +2$, in all other cases.\label{prop-secantBound3}
\end{enumerate}
\end{prop}
\begin{proof}
Case \ref{prop-secantBoundVero} follows from the construction in \cite[Proposition 5.2.2]{ed:si}. In case~(\ref{prop-secantBound2}), we may assume that $(a_1,a_2)=(1,1)$ by Lemma \ref{lem-nonmodSV}. Then the secant variety is the space of rank 3 matrices, which has the dimension indicated.
In general, and hence in case~(\ref{prop-secantBound3}), the dimension of the secant variety is bounded above by $2(\sum (n_i-1))+1$.
\end{proof}

\begin{eg}\label{333example} In the case of a Segre product with two factors and $n_1=3$, the set
\begin{align*}&x_{1,1} x_{2,1}, \ x_{1,1} x_{2,2},\ x_{1,2} x_{2,1} ,\ x_{1,2} x_{2,n_2},\ x_{1,3} x_{2,{n_2-1}} ,\ x_{1,3} x_{2,n_2} ,\\ &u_i:=x_{1,1} x_{2,i+1} - x_{1,2} x_{2,i},\  v_i:=x_{1,2} x_{2,i} - x_{1,3} x_{2,i-1},\ \quad  i=2,\dots,n_2-1
\end{align*}
is a separating set. Indeed, by induction on $n_2$ it suffices to show that the values of $x_{1,1} x_{2,3}, x_{1,2} x_{2,2},$ and $x_{1,3} x_{2,1}$ can be recovered from those of $x_{1,1} x_{2,1}, x_{1,1} x_{2,2},$ $x_{1,2} x_{2,1}, u_2, $ and $v_2$. If $x_{1,1} x_{2,1} \neq 0$, then one has $x_{1,2} x_{2,2} = \frac{ x_{1,1} x_{2,2} \cdot x_{1,2} x_{2,1} }{x_{1,1} x_{2,1}}$. If $x_{1,1} x_{2,1} = 0$ and $x_{1,2} x_{2,1} \neq 0$, then $x_{1,1} = 0$, so $x_{1,3} x_{2,1} = 0$, from which $x_{1,2} x_{2,2}$ and $x_{2,1} x_{2,3}$ can be obtained. The case $x_{1,1} x_{2,1} = 0$ and $x_{1,1} x_{2,2} \neq 0$ is similar. Finally, if $x_{1,1} x_{2,1} = x_{1,2} x_{2,1} = x_{1,1} x_{2,2} = 0$, then at most one of $x_{1,1} x_{2,3}, x_{1,2} x_{2,2},$ and $x_{1,3} x_{2,1}$ is nonzero. If one of these is nonzero, then the two of $u_2, v_2,$ and $x_{1,3} x_{2,1} - x_{1,1} x_{2,3}= - u_2 - v_2$ containing that monomial are equal to it (up to sign), while if all three monomials are zero, these three binomials are zero. Thus from the values of $u_2$ and $v_2$, one can determine which of $x_{1,1} x_{2,3}, x_{1,2} x_{2,2},$ and $x_{1,3} x_{2,1}$ is nonzero, and their value.\done
\end{eg}

%.................................

In case~(\ref{prop-secantBound2}) of Proposition \ref{prop-UBSV} the separating set satisfying the bound can be obtained by taking generic linear combinations of a generating invariant monomials. This is not true in general, as illustrated in Example \ref{eg-veroneses-example} above. A significant difference between the two cases is that the union of points belonging to secant lines is closed in case~(\ref{prop-secantBound2}) of Proposition \ref{prop-UBSV} but not in general, as for example for general Veronese varieties. With this in mind, we determine when the set of secant lines fills the secant variety of a Segre-Veronese variety. The following proposition is well-known in the case of Segre varieties; it translates to the fact that closest rank 2 approximation of a tensor is an ill-posed problem.

\begin{prop}\label{prop-notclosed} Let $X$ be the Segre-Veronese variety that is the image of the closed embedding $\prod_{i=1}^r \PP^{n_i-1} \hookrightarrow \PP^N$ given by the line bundle $\mcal{O}(a_1,\dots,a_r)$. If $r>2$ or $r=2$ and $(a_1,a_2)\neq(1,1)$, then the set of secant lines to $X$ does not fill the secant variety of $X$.
\end{prop}
\begin{proof} We will write vectors in tensor notation. Let $\phi_{a_i}$ denote the Veronese embedding of $\PP^{n_i-1}$ degree $a_i$, and set $e_{\alpha}$ to be the basis vector in the coordinate corresponding to the monomial with exponent $\alpha$ under the Veronese map.
First, let $r>2$. Set \[w = \phi_{a_4}(1,0,\dots,0) \otimes \cdots \otimes \phi_{a_r}(1,0,\dots,0)\,,\] 
and
\[
\begin{aligned}  v_{\lambda} = \lambda \cdot \phi_{a_1}(1, \lambda^{-1}, 0, \dots, 0) \otimes \phi_{a_2}(1, \lambda^{-1}, 0, \dots, 0) \otimes \phi_{a_3}(1, \lambda^{-1}, 0, \dots, 0) \otimes w\\
 -  \lambda \cdot \phi_{a_1}(1, 0, \dots, 0) \otimes \phi_{a_2}(1, 0, \dots, 0) \otimes \phi_{a_3}(1, 0, \dots, 0) \otimes w
 \end{aligned} \]
 for $\lambda \in \kk$,
 and 
 \[ 
\begin{aligned}
v_{\infty}= \ &e_{a_1-1,1,0,\dots,0}\otimes e_{a_2,0,0,\dots,0} \otimes  e_{a_3,0,0,\dots,0} \otimes w \\
&+ e_{a_1,0,0,\dots,0}\otimes e_{a_2-1,1,0,\dots,0} \otimes e_{a_3,0,0,\dots,0}\otimes w \\
&+e_{a_1,0,0,\dots,0}\otimes e_{a_2,0,0,\dots,0} \otimes  e_{a_3-1,1,0,\dots,0} \otimes w\,.
\end{aligned}
\]
 One may write 
\[\phi_{a_i}(1, \lambda^{-1}, 0, \dots, 0)=e_{a_1,0,0,\dots,0}+\lambda^{-1}e_{a_1-1,1,0,\dots,0}+\text{higher order terms in $\lambda^{-1}$}\,.\]
 We then have that $v_{\lambda}=v_{\infty} + \, \text{terms with negative powers of $\lambda$}$. One thus sees that $\{ v_\lambda \ | \ \lambda \in \kk\} \cup \{ v_{\infty} \}$ forms a locally closed subset in $\PP^N$. Clearly, $\{ v_\lambda \ | \ \lambda \in \kk\}$ is contained in the set of secant lines to $X$, but $v_{\infty}$ is a rank 3 tensor, see e.g., \cite{vds-lhl:tripblrap}, and hence is not contained in the secant set, but is in the secant variety.
 Now let $r=2$. Set
 \[ \begin{aligned} v_{\lambda}= \ &\lambda \cdot \phi_{a_1}(1, \lambda^{-1}, 0, \dots, 0) \otimes \phi_{a_2}(1, \lambda^{-1}, 0, \dots, 0) \\
 &-  \lambda \cdot \phi_{a_1}(1, 0, \dots, 0) \otimes \phi_{a_2}(1, 0, \dots, 0)
 \end{aligned}\]
for $\lambda \in \kk$, and
\[ v_{\infty}= e_{a_1-1,1,0,\dots,0} \otimes e_{a_2,0,  \dots,0} +  e_{a_1,0,  \dots,0} \otimes e_{a_2-1,1,0,\dots,0} \,,\]
 One again verifies that $\{ v_\lambda \ | \ \lambda \in \kk\} \cup \{ v_{\infty} \}$ is locally closed. It remains to show that $v_{\infty}$ does not lie on a secant line. In the case that $a_1=2, a_2=1$, this condition can be verified in Macaulay2 \cite{M2} by writing a system of equations for this vector to be expressed as the sum of two elements in $X$, and seeing that the ideal it generates is the trivial ideal. In the case of larger $a_i$, one sees that the coordinates corresponding to the  $\mathcal{O}(2,1)$ case give the same system of equations multiplied by a uniform scalar, and hence again have no solution.
 \end{proof}

When the secant set of a variety $X\subseteq \PP^N$ is not closed, by Lemma~\ref{lem-spec-iso-projection}, one may project from a point in $\Sec(X) \setminus \sigma(X)$; since this projection is not accounted for in the proof of Corollary~\ref{cor-upperbound}, one may hope that the bound given there can be harpened by one when the secant set is not closed. This is indeed the case in Examples~\ref{eg-veroneses-example}~and~\ref{333example}. This motivates the following.

\begin{conj}
We consider the Segre-Veronese variety which is the image of the closed embedding $\prod_{i=1}^r \PP^{n_i-1} \hookrightarrow \PP^N$ given by the line bundle $\mcal{O}(a_1,\dots,a_r)$. Then the minimal size of a separating set for the affine cone is bounded above by $2\sum_{i=1}^r n_i  - 2r +1$.
\end{conj}

%=================================================================
%=================================================================
%=================================================================

\section{Lower bounds on the size of separating sets}\label{section-LowerBounds}

In this section, we focus on the affine cones over Segre-Veronese varieties. We will give lower bounds for the sizes of separating sets. In most cases, these lower bounds agree with the upper bounds in the previous section, thus giving the precise cardinality of a minimal separating set. Our technique is based on the following observation and relies on the use of local cohomology (\cite{sbi-gjl-al-cm-em-aks-uw:thlc} provides a good reference).

\begin{lem}\label{lem-seprankLC}\cite[Section~3]{ed-jj:silc}
Let $G$ act linearly on $V$. Then the minimal size of a separating set for $G$ is bounded below by the maximum $i$ such that $\HH{i}{\I(\Ss_{V,G})}{\kvv}$ is nonzero.
\end{lem}

We will require two elementary lemmas on local cohomology. The second Lemma below, while well-known to experts, is proved here for lack of an appropriate reference.

\begin{lem}\label{lem-specialization}(see,~e.g.,~\cite[Theorem~9.6]{sbi-gjl-al-cm-em-aks-uw:thlc})
Let $I$ and $J$ be ideals in a noetherian ring $A$. Then 
\[{\cd(I,A)\gs \cd(I(A/J),A/J)}\,,\] where $\cd$ denotes the \emph{cohomological dimension}, that is, the greatest nonvanishing index of the local cohomology.
\end{lem}

\begin{lem}\label{lem-Kunneth} Let $A$ and $B$ be $\kk$-algebras, where $\kk$ is a field. Let $\A\subset A$ and $\B \subset B$ be ideals. Set $C=A \otimes_{\kk} B$ and $\C=\A C + \B C$. Then  $\HH{k}{\C}{C}\iso \bigoplus_{i+j=k} \HH{i}{\A}{A} \otimes_{\kk} \HH{j}{\B}{B}$. In particular, the cohomological dimension of $\C$ is the sum of the cohomological dimensions of $\A$ and of $\B$.
\end{lem}
\begin{proof} Let $\A = (f_1,\dots, f_s)$ and $\B=(g_1, \dots, g_t)$. One computes $\HH{k}{\C}{C}$ via the \v{C}ech complex $\{f_1,\dots, f_s, g_1, \dots, g_t\}$ on $C$, which we denote $\Ch(\{\underline{f},\underline{g}\},C)$. One verifies that we have an isomorphism of complexes \[\Ch(\{\underline{f},\underline{g}\},C)= \mathrm{Tot}\big(\Ch(\{\underline{f}\},A)\otimes_{\kk}\Ch(\{\underline{g}\},B)\big)\,.\]
As this is a tensor product of free modules (over $\kk$), the Kunneth formula yields an isomorphism
\[H^{\bullet}(\Ch(\{\underline{f},\underline{g}\},C))= \mathrm{Tot}\big(H^{\bullet}\Ch(\{\underline{f}\},A)\otimes_{\kk} H^{\bullet}\Ch(\{\underline{g}\},B)\big)\,.\]
As the \v{C}ech complexes $\Ch(\{\underline{f}\},A)$ and $\Ch(\{\underline{g}\},B)$ compute $\HH{i}{\A}{A}$ and $\HH{j}{\B}{B}$, the Lemma is established.
\end{proof}

We obtain the first main result of the section.

\begin{thm}\label{thm-segvarUB} Let $X$ be the Segre-Veronese variety corresponding to the bundle $\mathcal{O}(a_1,\dots,a_n)$. If at least one $a_i$ is not 1 or $p^e$, where $p=\mathrm{char}(\kk)$, then the minimal size of a separating set for the affine cone over $X$ is at least $2\sum_{i=1}^r n_i  - 2r +1$.
\end{thm}
\begin{proof} By Lemma~\ref{lem-nonmodSV} we may assume that $a_i\in\kk^*$ for all $i=1,\ldots,r$ and at least one $a_i$ is not 1. We will show that the cohomological dimension of    $\I(\Ss_{W,G})$ is at least $s={2\sum_{i=1}^r n_i  - 2r +1}$. By Proposition~\ref{prop-sepvarSV}, we can decompose $\Ss_{W,G}$ as a union of components isomorphic (via an automorphism of $W$) to $\Ss_{W,T}$. By Lemma~\ref{lem-equal}, and the assumption on the the $a_i$'s, there are at least two such components, and again by Proposition~\ref{prop-sepvarSV}, the intersection of any pair of distinct components is $\Nn_{W,T} \times \Nn_{W,T}$; note that this implies that the intersection of any union of components with another component is $\Nn_{W,T} \times \Nn_{W,T}$. Label the ideals of the distinct components as $\A_1,\dots,\A_t$, $\Nn=\I(\Nn_{W,T} \times \Nn_{W,T})$, and $\B_i=\A_1 \cap \cdots \cap \A_i$, so that $\Ss_{W,G}=\B_t$. There is a Mayer-Vietoris long exact sequence:
\[ \begin{aligned}
\cdots \longrightarrow \HH{i}{\Nn}{\kvv} \longrightarrow &\HH{i}{\B_j}{\kvv} \oplus \HH{i}{\A_{j+1}}{\kvv} \longrightarrow \HH{i}{\B_{j+1}}{\kvv} \\
\longrightarrow &\HH{i+1}{\Nn}{\kvv} \longrightarrow \HH{i+1}{\B_j}{\kvv} \oplus \HH{i+1}{\A_{j+1}}{\kvv}\longrightarrow \cdots \end{aligned}\]
We do not know the cohomological dimension of the ideals $\A_{j}$, so we argue by cases.

First, suppose that $\HH{s+1}{\A_j}{\kvv}\neq 0$, and its support is not just the homogeneous maximal ideal $\fm$. Then there is some prime $\p \subsetneq \fm$ for which $\HH{s+1}{\A_j}{\kvv}_{\p}\neq 0$. From Lemma~\ref{lem-Nvg_Nvg} below, we know that $\HH{s+2}{\Nn}{\kvv}_{\p}=0_\p=0$, and  $\HH{s+1}{\Nn}{\kvv}_{\p}\cong \HH{d}{\fm}{\kvv}_\p=0$, since $\HH{d}{\fm}{\kvv}$ is artinian, and hence its support is $\{\fm\}$; see, e.g., \cite[Exercise~7.7]{sbi-gjl-al-cm-em-aks-uw:thlc}. 
 Consequently, there are isomorphisms $\HH{s+1}{\B_{j+1}}{\kvv}_{\p}\cong \HH{s+1}{\B_{j}}{\kvv}_{\p} \oplus \HH{s+1}{\A_{j+1}}{\kvv}_{\p}$ for all $j>0$. By induction, we see that these modules are nonzero for each $j$, so the cohomological dimension of $\I(\Ss_{W,G})$ is at least $s+1$ in this case.

Second, suppose that $\HH{s+1}{\A_j}{\kvv}\neq 0$, and that its support is the homogeneous maximal ideal $\fm$. By an argument similar to the previous case, we see that $\HH{s+1}{\B_j}{\kvv}$ is supported on the maximal ideal for each $j$. 
By \cite[Corollary~3.6]{Lyu} in characteristic zero and \cite[Corollary~3.7]{HS} in positive characteristic, there are isomorphisms $\HH{s+1}{\A_j}{\kvv}\cong \HH{d}{\fm}{\kvv}^{\oplus a}$ and $\HH{s+1}{\B_j}{\kvv}\cong \HH{d}{\fm}{\kvv}^{\oplus {b_j}}$ for some positive integers $a$,$b_j$, where $d=\dim{\kvv}$. By \cite[Theorem~1.1]{MZ}, these isomorphisms are degree-preserving. Since $[\HH{d}{\fm}{\kvv}]_{-d}\cong \kk$, by restricting to degree $-d$ and applying Lemma~\ref{lem-Nvg_Nvg} below, we obtain from the Mayer-Vietoris sequence right-exact sequences
\[ \kk \longrightarrow \kk^a \oplus \kk^{b_j} \longrightarrow \kk^{b_{j+1}} \longrightarrow 0. \]
It follows by an easy induction that $b_j\neq 0$ for each $j$, and the cohomological dimension of $\I(\Ss_{W,G})$ is at least $s+1$ in this case.

Finally, suppose that $\HH{s+1}{\A_j}{\kvv}=0$. It follows from induction on $j$ using the Mayer-Vietoris sequence above and vanishing of $\HH{s+2}{\Nn}{\kvv}$ by Lemma~\ref{lem-Nvg_Nvg} below, that $\HH{s+1}{\B_j}{\kvv}=0$ for all $j$. Then we have that $\HH{s}{\B_j}{\kvv}$ surjects onto $\HH{s+1}{\Nn}{\kvv}$ for all $j$, so $\HH{s}{\B_j}{\kvv}\neq 0$. In particular, the cohomological dimension of $\I(\Ss_{W,G})$ is at least $s$ in this case. This case concludes the proof.
\end{proof}

\begin{lem}\label{lem-Nvg_Nvg} Let $Y$ be the affine cone over the nonmodular Segre-Veronese variety $\prod_{i=1}^r \PP^{n_i-1} \hookrightarrow \PP^N$ given by the line bundle $\mathcal{O}(a_1,\dots,a_r)$. Set $t=\sum_{i=1}^r n_i -r$ and $\Nn=\I(\Nn_{W,G} \times \Nn_{W,G})$. Then, 
\[ \HH{j}{\Nn}{\kvv}=0 \ \text{for} \ j>2t+2\,, \ \text{and} \ \HH{2t+2}{\Nn}{\kvv}\cong \HH{d}{\fm}{\kvv}\,,\]
 where $\fm$ is the homogeneous maximal ideal of $\kvv$, and $d=\dim \kvv$.
\end{lem}

\begin{proof} First, we compute the local cohomology with support in $\I(\Nn_{W,G})$ in $\kv$. We have that
\[ \begin{aligned}
 &\I(\Nn_{W,G})= \\
 &\hspace{1cm} \big(\ M_1 \cdots M_r \ \big| \ M_i \ \text{is a monomial of degree $a_i$ in the variables} \  x_{i1}, \dots , x_{i n_i} \big)\,,\end{aligned}\]
 whose radical is
 \[J:=(\, x_{1,j_1}\cdots x_{r,j_r} \mid   j_i\in[n_i]\,) = \prod_{i=1}^r (x_{i,1}\, \dots, x_{i,n_i})\,.\]
Note that this coincides with the defining ideal for the nullcone of the action of the torus $T$ (see Proposition \ref{prop-sepvarSV}).
 
Set $\A_i=(x_{i,1}, \dots, x_{i,n_i})$. We apply the Mayer-Vietoris spectral sequence of \cite{jam-rgl-az:lcasmi}. Since each $A \subseteq \{1,\dots,r\}$ yields a distinct ideal $\A_A=\sum_{i\in A} \A_i$, the intersection poset of the subspace arrangement defined by $J$ is the full Boolean poset. Thus, the associated simplicial complex of each interval in the poset is a homology sphere of dimension $\# A -2$, where, by convention, the $(-1)$-sphere is the empty set. Set $n_A=\sum_{i\in A} n_i = \mathrm{ht} (\A_A)$. By \cite[Corollary~1.3]{jam-rgl-az:lcasmi}, there is a filtration of the local cohomology with support in $J$ such that the associated graded module satisfies
\[ \mathrm{gr} \big( \HH{q}{J}{\kv} \big) \iso \bigoplus_{\emptyset \neq A\subseteq \{1,\dots,r\}} \HH{n_A}{\A_A}{\kv} \otimes_{\kk} \tilde{H}_{n_A - q -1} ( S^{\# A -2} , \kk)\,.  \]
Thus, the cohomological dimension of $\I(\Nn_{V,G})$ is 
\[ \max \big\{ n_A - \# A +1 \ | \ \emptyset \neq A\subseteq [r] \big\} =  t +1 \,,\]
and there is an isomorphism $\HH{t+1}{\I(\Nn_{V,G})}{\kv}\cong \HH{d/2}{\fn}{\kv}$, where $\fn=\A_{\{1,\dots,r\}}$ is the homogeneous maximal ideal of $\kv$.
The statement of the Lemma then follows from Lemma~\ref{lem-Kunneth}.
\end{proof}

For Segre varieties with at least three factors, we obtain a lower bound, but we do not expect this bound to be sharp in general.

\begin{prop}\label{prop-LBS}
 Let $X$ be the Segre variety that  is the image of the closed embedding $\prod_{i=1}^r \PP^{n_i-1} \hookrightarrow \PP^N$ given by the line bundle $\mcal{O}(1,\dots,1)$, with $r>2$. Suppose, without loss of generality that $n_1\ls n_2\ls \cdots \ls n_r$. The size of a separating set for the affine cone over $X$ is at least $2 \big( \sum_{i=2}^{r} n_i \big) - 2r +4$.
 \end{prop}

\begin{proof}
As $X$ is nonmodular and the $a_i$'s are all 1, its ring of homogeneous coordinates coincides with the ring of invariants of an action of a torus of rank $r-1$ as discussed in Section \ref{section-SV}. The following ideal cuts out the corresponding separating variety:
\[ I:=\big(\ M_1 \cdots M_r \otimes 1-1\otimes M_1 \cdots M_r \ \big| \
  M_i \ \text{is one of the variables} \  x_{i1}, \dots , x_{i n_i} \ \big).
 \]

 By Lemma \ref{lem-specialization}, the cohomological dimension of $I$ is bounded below by the cohomological dimension of the ideal we obtain via the linear specialization to $x_{1,1}\otimes 1=x_{1,2}\otimes 1=1\otimes x_{1,1}=1$ and $1\otimes x_{1,1}=0$. This ideal is
 \[
 \big(\ M_2 \cdots M_r \otimes 1,1\otimes M_2 \cdots M_r \ \big| \
 M_i \ \text{is one of the variables} \  x_{i1}, \dots , x_{i n_i} \ \big)
\,.\]
This ideal coincides with $\I(\Nn_{\widetilde{W},\widetilde{T}} \times \Nn_{\widetilde{W},\widetilde{T}})$ for the action of the torus $\widetilde{T}$ of rank $r-2$ on $\widetilde{W}:=\spn_{\kk}\{w_0,w_{i,j_i}\mid i\gs 2\}$ defining the Segre embedding of $\prod_{i=2}^r \PP^{n_i-1} \hookrightarrow \PP^N$. Applying Lemma~\ref{lem-Nvg_Nvg}, we obtain the bound in the statement.
 \end{proof}
 
 The proof above works for general Segre-Veronese varieties. We have restricted the statement in the proposition above, because in all other cases we can obtain a more precise result. The only remaining case is that of Segre products with two factors. In this case, local cohomology groups fail to provide a sufficient obstruction in positive characteristic, but we may argue along similar lines using \'etale cohomology. We refer the reader to \cite{Milne} for the facts from \'etale cohomology used below.
 
 Fix $\L=\ZZ/q\ZZ$, where $q\neq \chara\kk$ is a prime.

\begin{prop}\label{prop-upper} If $Y$ is a $d$-dimensional variety that is covered by $k$ affines, then $\Het{i}{Y,\L}=0$ for all $i\gs d+k$. In particular, if $Z$ is a closed subset of $\AA^d$ and $\Het{d+k-1}{\AA^d \setminus Z,\L}\neq 0$, then $Z$ cannot be defined by fewer than $k$ equations.
\end{prop}

We will use the following result.

\begin{prop}[Bruns-Schwanzl \cite{wb-rs:tneddv}] Let  $M$ be a $2 \times s$ matrix of indeterminates in the polynomial ring $A$. Set $Z=V(I_2(M))\subset \AA^{2s}$. Then $\Het{4s-4}{\AA^{2s}\setminus Z, \L}\iso \L$, and the higher \'etale cohomology groups vanish.
\end{prop}

\begin{thm} \label{thm-S2factor} For the affine cone over the Segre embedding of $\PP^{n_1-1} \times \PP^{n_2-1}$, any separating set has size at least $2n_1+2n_2-4$.
\end{thm}
\begin{proof}
By the long exact sequence
\[ \cdots\rightarrow\Het{r}{\AA^{2s}, \L} \rightarrow\Het{r}{\AA^{2s}\setminus Z, \L} \rightarrow \Hetl{r+1}{Z}{\AA^{2s}, \L} \rightarrow \Het{r+1}{\AA^{2s}, \L} \rightarrow \cdots\]
it follows that $\Hetl{4s-3}{Z}{\AA^{2s}, \L}\iso \L$ and the higher such groups vanish. Then, the Gysin isomorphism yields $\Het{2s-1}{Z,\L}\iso \L$, with the higher ones vanishing. Another application of the Gysin isomorphism yields $\Hetl{4s+2t-3}{Z\times \{0\}}{\AA^{2s}\times \AA^{t},\L}\iso \L$, where $0$ is the origin in $\AA^t$. Applying the sequence above, we obtain
\[\Het{4s+2t-4}{(\AA^{2s}\times \AA^{t})\setminus (Z\times \{0\}),\L}\iso \L\]
and the higher groups vanish.

By Proposition \ref{prop-sepvarSV}, part \ref{prop-sepvarS2}, the separating variety decomposes as
\[
 \Ss_{W,T}=\overline{\Gamma_{W,T}}\cup W_{\hat{1}}\times W_{\hat{1}} \cup W_{\hat{2}}\times W_{\hat{2}},
\]

Write $D,X,Y$ to denote $\overline{\Gamma_{W,T}},W_{\hat{1}},W_{\hat{2}}$, respectively, $A=\AA^{2(m+n)}$, and $S$ for the separating variety $W\cup X \cup Y$. For all sequences below, we consider \'etale cohomology with coefficients in $\L$.

We obtain one Mayer-Vietoris sequence:
\[\cdots\rightarrow\Het{i}{A\setminus\{0\}}\rightarrow \Het{i}{A\sm X}\oplus\Het{i}{A \sm Y}\rightarrow \Het{i}{A \sm (X \cup Y)} \rightarrow \Het{i+1}{A\setminus\{0\}}\rightarrow\cdots\]
from which we conclude that
\[ \Het{i}{A\sm X}\oplus\Het{i}{A \sm Y}\iso \Het{i}{A \sm (X \cup Y)}\]
by the natural inclusion maps for $i<4n_1+4n_2-1$.

From the Mayer-Vietoris sequence:
\begin{align*}
\cdots\rightarrow\Het{i}{A\setminus\{0\}}\rightarrow&\Het{i}{A\setminus(D\cap X)} \oplus \Het{i}{A\setminus(D\cap Y)} \\
&\rightarrow \Het{i}{A\sm (D\cap (X\cup Y))}
\rightarrow\Het{i+1}{A\setminus\{0\}}\rightarrow\cdots\end{align*}
we conclude
 \[\Het{i}{A\setminus(D\cap X)} \oplus \Het{i}{A\setminus(D\cap Y)} \iso \Het{i}{A\sm (D\cap (X\cup Y))}\] for $i<4n_1+4n_2-1$.

We consider one more Mayer-Vietoris sequence:
\begin{align*}
\cdots\rightarrow\Het{i}{A\sm (D\cap (X\cup &Y))}\rightarrow\Het{i}{A \sm D} \oplus \Het{i}{A \sm (X \cup Y)} \\&\rightarrow \Het{i}{A \sm S}
\rightarrow\Het{i+1}{A\sm (D\cap (X\cup Y))}\rightarrow\cdots
\end{align*}
which, assuming $n_1\gs 3$, applying the consequences of the long exact sequences above also reads
\begin{align*}
\cdots\rightarrow\Het{i}{A&\setminus(D\cap X)} \oplus \Het{i}{A\setminus(D\cap Y)}\rightarrow\Het{i}{A \sm D} \\&\rightarrow \Het{i}{A \sm S}
\rightarrow\Het{i+1}{A\setminus(D\cap X)} \oplus \Het{i+1}{A\setminus(D\cap Y)}\rightarrow\cdots
\end{align*}
for $4n_2<i<4n_2+4n_2-1$. In particular, for $t=4n_2+4n_2-3$,
\[
\cdots \rightarrow\Het{t-1}{A\sm S}\rightarrow \Het{t}{A\setminus(D\cap X)} \oplus \Het{t}{A\setminus(D\cap Y)}
\rightarrow\Het{t}{A \sm D}
\rightarrow\cdots
\]
which computes as
\[
\cdots \rightarrow\Het{t-1}{A\sm S}\rightarrow \L\times\L
\rightarrow\L
\rightarrow\cdots
\]
Thus, $\Het{4n_1+4n_2-4}{A\sm S}\neq 0$. The theorem then follows from Proposition~\ref{prop-upper}.
\end{proof}

\begin{proof}[Proof of Theorem~\ref{thm-MinSizeSV}]
Recall that we can reduce to the case of nonmodular Segre-Veronese varieties by Lemma~\ref{lem-nonmodSV}. The upper bounds are given by Proposition \ref{prop-UBSV} and the lower bounds are given by \cite[Main Theorem]{ed-jj:silc} for Case~(\ref{thm-MinSizeVero}), Theorem~\ref{thm-segvarUB} for Case~(\ref{thm-MinSizeSV1}), Proposition~\ref{prop-LBS} for Case~(\ref{thm-MinSizeSV3}), and Theorem~\ref{thm-S2factor} for Case~(\ref{thm-MinSizeSV2}). In Cases~(\ref{thm-MinSizeVero})~and~(\ref{thm-MinSizeSV2}), the upper and lower bounds coincide.
\end{proof}

%------------------------------------------------------------------------

%------------------------------------------------------------------------

%------------------------------------------------------------------------

\section{Monomial separating sets}\label{section-MonomialSepSets}

The focus of this section is on the invariants of representations of tori and their monomial separating sets. We include as a special case those representations whose ring of invariants is isomorphic to the ring of homogeneous coordinates on Segre-Veronese varieties.

\subsection{Combinatorial characterization of monomial separating subalgebras}

A monomial subalgebra of the invariant ring will be given by a subsemigroup $\Ss\subseteq \Ll$. For $I\subseteq \{1,\ldots,n\}$, set $\Ss_I=\Ll_I\cap \Ss$.

\begin{prop}\label{prop-SsepCharp}
 Suppose $\kk$ has positive characteristic $p$. The subsemigroup ${\Ss\subseteq \Ll}$ gives a separating algebra if and only if there exist $m\gs 1$ such that $p^m\Ll\subseteq \Ss$.
\end{prop}
\begin{proof}
 Suppose there exist $m\gs 1$ such that $p^m\Ll\subseteq \Ss$. Take $u,v\in V$ and suppose they are separated by some invariant. Without loss of generality we may assume they are separated by $x^\alpha$ with $\alpha\in \Ll$. By our assumption, $p^m\alpha=\gamma\in \Ss$. Then $x^\gamma$ separates $u$ and $v$. Indeed, otherwise we have \[(x^\alpha(u))^{p^m}=x^\gamma(u)=x^\gamma(v)=(x^\alpha(v))^{p^m}\,,\]
 and so $x^\alpha(u)=x^\alpha(v)$, a contradiction.
 
 Now suppose $\Ss$ gives a separating algebra $A$. As this algebra is a graded subalgebra, it follows that the extension $A\subseteq \kx^T$ is finite and $\kx^T$ is the purely inseparable closure of $A$ in $\kx$ (see \cite[Remark 1.3]{hd-gk:ciagac} or \cite[Theorem 4]{fdg:viac}). Hence, for any $x^\alpha\in\kx^T$, there exist $m_\alpha\in \NN$ such that $(x^\alpha)^{p^{m_\alpha}}\in A$. The finiteness of the extension ensures that there exists a natural number $m$ such that $(x^\alpha)^{p^{m}}\in A$ for all $\alpha\in\Ll$. It follows that $p^m\Ll\subseteq\Ss$.
\end{proof}

\begin{lem} \label{lem-SSep}
 Suppose that for all $I\subseteq [n]$, $\Ll_I\subseteq \ZZ\Ss_I$, and for all $\alpha\in \Ll$, $\alpha_i\neq 0$ implies that there exist $\gamma\in \Ss$ such that  $i\in\supp (\gamma)$ and $\supp (\gamma) \subseteq \supp (\alpha)$. Then $\Ss$ gives a separating algebra.
\end{lem}
\begin{proof}
 Take $u,v\in V$ and suppose they are separated by some invariant, without loss of generality, suppose they are separated by $x^\alpha$ with $\alpha\in \Ll$. Suppose first that $x^\alpha(u)=0\neq x^\alpha(v)$. As $x^\alpha(u)=0$, there exists $i\in \supp (\alpha)$ such that $u_i=0$. By our assumption, there exist $\gamma\in \Ss$ such that $i\in \supp (\gamma)$ and $\supp (\gamma)\subseteq \supp (\alpha)$. Then as $v_j\neq 0$ for all $j\in \supp (\alpha)$, we have $x^\gamma(u)=0\neq x^\gamma (v)$.
 
 Suppose now that both $x^\alpha(u)$ and $x^\alpha(v)$ are nonzero. Then $u_i$ and $v_i$ are nonzero for all $i\in\supp (\alpha)$. By our assumption, there exist $\gamma,\gamma'\in \Ss$ such that $\alpha=\gamma-\gamma'$. Then one of $x^\gamma$ or $x^{\gamma'}$ must separate $u$ and $v$. Indeed, otherwise we have
 \[x^\alpha(u)=\frac{x^\gamma(u)}{x^{\gamma'}(u)}=\frac{x^\gamma(v)}{x^{\gamma'}(v)}=x^\alpha(v)\,,
 \]
a contradiction.
\end{proof}

\begin{prop}\label{prop-sepalg}
 Suppose $\kk$ has characteristic zero. Then the following are equivalent:
 
\begin{enumerate}
 \item $\Ss$ gives a separating algebra. \label{prop-sepalg1}
 
 \item For all $I\subseteq [n]$, $\Ll_I\subseteq \ZZ\Ss_I$, and for all $\alpha\in \Ll$, $\alpha_i\neq 0$ implies that there exist $\gamma\in \Ss$ such that  $i\in\supp (\gamma)$ and $\supp (\gamma) \subseteq \supp (\alpha)$. \label{prop-sepalg2}
 
 \item For any prime number $p$ there exist $m\gs 1$ such that  $p^m\Ll\subseteq \Ss$. \label{prop-sepalg3}
 
 \item There exist prime numbers $p,q$ and an $m\gs 1$ such that $p^m\Ll\subseteq \Ss$ and $q^m\Ll\subseteq \Ss$. \label{prop-sepalg4}

\end{enumerate}
\end{prop}
\begin{proof}
 (\ref{prop-sepalg1})$\Rightarrow$ (\ref{prop-sepalg2}): Suppose that $\Ss$ gives a separating algebra. As $T$ is reductive, the restriction of any separating set to a $T$-stable subspace gives a separating set \cite[Theorem 2.3.16]{hd-gk:cit}. Hence for any subset $I\subseteq [n]$, $\Ss_I$ must give a separating algebra in $\kk[V_I]^T$. As we assume $\kk$ has characteristic zero, the field of fractions of the separating algebra given by $\Ss_I$ coincides with the field of fractions  of the invariant ring \cite[Proposition 2.3.10]{hd-gk:cit}, that is $\ZZ\Ss_I=\ZZ\Ll\supseteq \Ll$. Now take $\alpha\in \Ll$ and suppose $i_0\in \supp (\alpha)$. Consider the points $u,v\in V$ defined as
 \begin{align*}
    u_i &=  \left\{\begin{array}{cl}
		1 & i\in \supp (\alpha)\\
		0 & \text{otherwise}\end{array}\right. \\
    v_i &=  \left\{\begin{array}{cl}
		1 & i\in \supp (\alpha)\setminus \{i_0\}\\
		0 & \text{otherwise}\,. \end{array}\right.
   \end{align*}
Then $u,v\in V_{\supp (\alpha)}$ and $x^\alpha(u)=1\neq 0=x^\alpha(v)$. As $\Ss_{\supp (\alpha)}$ gives a separating algebra, there exist $\gamma \in \Ss_{\supp (\alpha)}$ such that $x^\gamma(u)\neq x^\gamma(v)$. It follows that $i_0\in\supp (\gamma)$, since otherwise $x^\gamma(u)=1=x^\gamma(v)$, a contradiction.

(\ref{prop-sepalg2})$\Rightarrow$ (\ref{prop-sepalg3}): The integer matrix $A$ gives a representation of the torus of rank $r$ over any field. Condition (\ref{prop-sepalg2}) does not involve the base field at all. Thus, if (\ref{prop-sepalg2}) holds, we can think of it as holding over a field of any characteristic $p$. Then by Lemma \ref{lem-SSep}, $\Ss$ gives a separating algebra over any field, and by Proposition \ref{prop-SsepCharp}, it follows that for any prime $p$ there exists $m\in\NN$ such that $p^m\Ll\subseteq \Ss$.

(\ref{prop-sepalg3})$\Rightarrow$ (\ref{prop-sepalg4}): Immediate.

(\ref{prop-sepalg4})$\Rightarrow$ (\ref{prop-sepalg1}): Take $u,v\in V$ and suppose they are separated by some invariant, without loss of generality we may assume they are separated by $x^\alpha$ with $\alpha\in \Ll$. By our assumption, $p^{m}\alpha =\gamma_1\in \Ss$ and $q^{m}\alpha =\gamma_2\in \Ss$. If $x^{\alpha}(u)=0\neq x^\alpha(v)$, then $x^{\gamma_1}(u)=(x^\alpha(u))^{p^{m}}=0\neq (x^\alpha(v))^{q^{m}}=x^{\gamma_2}(v)$. So we may suppose both $x^\alpha(u)$ and $x^\alpha(v)$ are nonzero. One of $x^{\gamma_1}$ or $x^{\gamma_2}$ must separate $u$ and $v$. Indeed, otherwise $x^{\alpha}(u)/x^\alpha(v)$ is both a $p^{m}$-th and a $q^{m}$-th root of unity, that is, $x^{\alpha}(u)/x^\alpha(v)=1$, a contradiction.

\end{proof}

%------------------------------------------------------------------------

\subsection{Invariants with small support separate}

We apply the results of the previous section to give a bound, for a torus action $T$, on the number $r$ such that there exists a separating set consisting of elements that each involve at most $r$ variables. In \cite{md-es:hdag}, Domokos and Szab\'o define invariants of algebraic groups to bound this number for actions of algebraic groups on product varieties; in fact, one may ask this question for any subring of a polynomial ring, so that one has a well-defined notion of the number of variables an element involves.

We do not believe that the following two results are new, but we have not seen the exact statement of Theorem \ref{thm-r+1} in the literature.

\begin{lem}\label{lem-sparsevectors} Let $\kk$ be a field, $n > m$ and set $T_m=\{I \subseteq [n] \ | \ |I|=m\}$. Let $U=\{v_I \ | \ I \in T_m\}$ be a collection of nonzero vectors in $\kk^n$ such that the projection of $v_I$ onto the coordinate subspace $\kk^I$ is zero. Then $\dim(\spn(U))\gs m+1$.
\end{lem}
\begin{proof} We proceed by induction on $m$. For the case $m=0$, $T_0=\{ \emptyset \}$ and $v_{\emptyset}$ is a nonzero vector by hypothesis; the projection onto $\kk^{\emptyset}\cong \kk^0$ is zero for all vectors. For the inductive step, assume without loss of generality that the first coordinate of $v_{[m]}$ is nonzero. By the inductive hypothesis, there are $n$ linearly independent vectors $\{w_1, \dots, w_m\}$ in $\{v_I \ | \ I \subseteq [n], |I|=m, 1\in I\}$, since omitting the first coordinate produces a set of vectors satisfying the statement of the lemma. As the vectors $\{w_1, \dots, w_m\}$ all have first coordinate zero, the set of vectors $\{w_1, \dots, w_m, v_{[m]}\}$ is linearly independent.
\end{proof}

\begin{thm} Let $A$ be a surjective $r \times n$ matrix with $n > r$. Then the lattice $\ker_\ZZ(A)\subset \ZZ^n$ is generated by elements with at most $r+1$ nonzero entries. \label{thm-r+1}
\end{thm}
\begin{proof} For a subset $I\subseteq [n], |I|=r+1$, let $A_I$ be the matrix obtained from $A$ by taking only the columns whose indices lie in $I$. Let $K_I\subseteq \ZZ^I \subseteq \ZZ^n$ be the kernel of $A_I$, and $K$ be the kernel of $A$. By definition, $\ZZ \langle K_I \rangle \subseteq K$; we will show that these lattices agree upon tensoring with $\QQ$ and $\FF_p$ for each prime $p$. Note that for each such $I$, $K_I$ contains a nonzero vector. Moreover, $K_I$ contains a vector that is nonzero mod $p$ for each prime $p$: if $p v_I \in K_I$, then $v_I\in K_I$ as well. Thus, for any $\kk=\QQ, \FF_p$, we have $\ZZ \langle K_I \rangle \otimes \kk \subseteq K\otimes\kk$, and $\ZZ \langle K_I \rangle \otimes \kk$  satisfies the hypotheses of Lemma~\ref{lem-sparsevectors} with $m=n-(r+1)$, so that its dimension is at least $n-r$. As the sequence
\[ 0 \rightarrow K \rightarrow \ZZ^n \stackrel{A}{\rightarrow} \ZZ^r \rightarrow 0 \]
is split, the dimension of $K \otimes \kk$ is $n-r$. Thus, $\ZZ \langle K_I \rangle \otimes \kk= K\otimes \kk$ for all such $\kk$, so $\ZZ \langle K_I \rangle = K$ as required.
\end{proof}

\begin{cor}\label{cor-r+1}
 Let $V$ be a $n$-dimensional representation of a torus $T$ of rank $r\ls n$. The rational invariants $\kk(V)^T$ are generated by rational invariants each involving at most $r+1$ variables.
\end{cor}

\begin{thm}\label{thm-2r+1}
  Let $V$ be a $n$-dimensional representation of a torus $T$ of rank $r\ls n$. The invariants involving at most $2r+1$ variables form a separating set.
\end{thm}
\begin{proof}
Let $\Ss\subseteq \Ll$ be the subsemigroup generated by all elements with support of size at most $2r+1$. We will show that $S$ satisfies the conditions of Lemma~\ref{lem-SSep}.

First we show that for all $\alpha\in \Ll$ with $\alpha_{i}\neq 0$, there exist $\gamma\in \Ss$ such that  $i\in\supp (\gamma)$ and $\supp (\gamma) \subseteq \supp (\alpha)$.
 Take $\alpha\in \Ll$, and suppose $\alpha_{i_0}\neq 0$. If ${|\supp (\alpha)  | \, \ls 2r+1}$, then ${\alpha \in \Ss}$, so there is nothing to do. So we suppose that ${|\supp (\alpha) | \, > 2r+1}$. We can rewrite the equation $0=\sum_{i=1}^n\alpha_i m_i$ as
 \[\frac{-\alpha_{i_0}}{\sum_{i\in\supp (\alpha)\setminus\{i_0\}}\alpha_i}m_{i_0}=\sum_{i\in\supp (\alpha)\setminus\{i_0\}} \left(\frac{\alpha_i}{\sum_{i\in\supp (\alpha)\setminus\{i_0\}}\alpha_i}\right) m_i\,,\]
 and so $(\nicefrac{-\alpha_{i_0}}{\sum_{i\in\supp (\alpha)\setminus\{i_0\}}\alpha_i})m_{i_0}$ is in the convex hull of $\{m_i\mid i\in \supp (\alpha)\setminus\{i_0\}\}$. By Carath\'eodory's Theorem \cite{cc:uvkpgwna} there is a subset $K\subseteq \supp (\alpha)\setminus\{i_0\}$ of size $|K| \ls r+1$ such that $(\nicefrac{-\alpha_{i_0}}{\sum_{i\in\supp (\alpha)\setminus\{i_0\}}\alpha_i})m_{i_0}$ is in the convex hull of $\{m_k \mid k\in K\}$. Hence, we have an equation
\[\frac{-\alpha_{i_0}}{\sum_{i\in\supp (\alpha)\setminus\{i_0\}}\alpha_i}m_{i_0}=\sum_{k\in K} \delta_k m_k\,,\]
where $\delta_k\gs 0$ and $\sum_{k\in K}\delta_k = 1$.  Multiplying by a sufficiently large natural number, we find $\sum_{i\in K\cup \{i_0\}}\gamma_im_i=0$ with $\gamma_i\in\NN$. Define $\gamma\in \NN^n$ as follows:
 \[
  \gamma_i=\left \{ \begin{array}{cl}
                    \gamma_{i_0} & \mathrm{ if }~ i=i_0\\
                    \gamma_i     & \mathrm{ if }~ i\in K\\
                    0           & \mathrm{ otherwise}
                   \end{array}\right.\]
Then $\gamma_{i_0}\neq 0$ and $\gamma$ has support $K\cup \{i_0\}\subseteq \supp (\alpha)$ of size at most $r+2\ls 2r+1$ so that $\gamma \in \Ss$ as required.

Now we show that for all $I\subseteq [n]$ we have $\Ll_I\subseteq \ZZ\Ss_I$. 
Fix $I\subseteq [n]$. If $\Ll_I=0$, we are done, so suppose $\Ll_I\neq 0$. Take $\alpha\in \Ll_I$. Set $I'=\supp (\alpha)$. By Corollary~\ref{cor-r+1}, we can write $\alpha$ as a $\ZZ$-linear combination of elements of $\ker_\ZZ A_{I'}$ with support of size at most $r+1$. It will suffice to show that any $\beta\in \ker_\ZZ A_{I'}$ with $|\supp \beta|\ls r+1$ can be written $\beta=\gamma-\gamma'$ with $\gamma,\gamma'\in\Ll_{I'}$ having support of size at most $2r+1$.

Take $\beta=\beta^+-\beta^-\in\ker_\ZZ A_{I'}$ with $\beta^+,\beta^-\in\NN^n$ with disjoint support and $|\supp (\beta)| \ls r+1$. Set $J^+=\supp (\beta^+)$ and $J^-=\supp (\beta^-)$. Note that we have ${|J^+| + |J^-|\ls r+1}$ and without loss of generality, both $J^+$ and $J^-$ are nonempty, and so ${\max{\{|J^+|,|J^-|\}}\ls r}$. As $\alpha$ has full  support $I'$, $0$ is an interior point of the convex hull of the weight vectors $\{m_i \, | \, i \in I'\}$, that is, there exists an equation of the form
\begin{equation}\sum_{i\in J^-} \lambda_i m_i + \sum_{j \notin J^-} \lambda_j m_j = 0\,,\label{eqn-invfullsupp}\end{equation} 
with $\lambda_i>0$ and $\sum_{i\in I'}\lambda_i=1$.

Set 
\[m'=-\left(\frac{1}{\sum_{j\notin J^-}\lambda_j}\right)\left(\sum_{i\in J^-} \lambda_i m_i \right) \quad \text{and} \quad \lambda_i'=\left(\frac{1}{\sum_{j\notin J^-}\lambda_j}\right)\lambda_i\,.\]
Then Equation~(\ref{eqn-invfullsupp}) can be rewritten as
\[m'=\sum_{i\notin J^-}\lambda_i'm_i\,.\]
Note that $\lambda_i'>0$ and $\sum_{i\notin J^-}\lambda_i'=1$, so that $m'$ is an interior point of the convex hull of $\{ m_i \, | \, i \notin J^-\}$. As $J^+$ is nonempty and disjoint from $J^-$, there exist $j_0\in J^+\setminus J^-$. By Watson's Carath\'eodory Theorem \cite{dw:artkc}, there exists a subset $K\subseteq I'$  with $|K|\ls r$ and $K \cap J^- = \emptyset$ such that $m'$ is in the convex hull of $\{m_{j_0}\} \cup \{ m_k \, | \, k \in K\}$, that is, there are nonnegative rational numbers $\mu_{j_0}, \mu_k$ such that $\sum_{k\in K\cup\{j_0\}}\mu_k=1$ and
\begin{equation} m'=\mu_{j_0}m_{j_0} + \sum_{k\in K} \mu_{k} m_{k}\,.\label{eqn-smallsupp}\end{equation} 
Substituting $m'$ for its value and reorganizing, we then get an equation
\[\sum_{i\in J^-}\left(\frac{\lambda_i}{\sum_{j\notin J^-}\lambda_j}\right)m_i  + \mu_{j_0} m_{j_0} + \sum_{k\in K}\mu_km_k=0\,.\]
It follows that there is $\gamma\in\Ll_I$ with support $J^-\cup \{j_0\}\subseteq \supp (\gamma) \subseteq J^-\cup\{j_0\}\cup K$. Note that $|\supp (\gamma) | \ls |J^-|+1+|K|\ls r+r+1=2r+1$. We may assume $\gamma-\beta^-\in\NN^n$, multiplying $\gamma$ by a large natural number if needed, and so $\beta+\gamma$ has support 
\[\supp(\beta+\gamma)\subseteq \supp (\beta) \cup \supp( \gamma) \subseteq J\cup (J^-\cup\{j_0\}\cup K)=J\cup K\,.\]
It follows that $|\supp (\beta+\gamma)|  \ls |J|+|K|\ls r+1+r=2r+1.$ Thus, we can write $\beta=(\beta+\gamma)-\gamma$ as a difference of elements of $\Ll_I$, with support of size at most $2r+1$.
\end{proof}

%''''''''''an example

\begin{eg}[The bound given by Theorem \ref{thm-2r+1} is sharp]
 We consider the $(2r+1)$-dimensional representation of the torus of rank $r$ given by the matrix of weights:
 \[A:=\left(\begin{array}{c|c|c}
             I & \begin{array}{c}
             5\\
             \vdots\\
             5
                  \end{array} & -6I
            \end{array}\right).\]
As $A$ is already in reduced echelon form,
\[\ker_\ZZ A=\Big\langle v_0=\Big(-5\sum_{j=1}^r e_j, 1, 0\Big), v_i=\big(6e_i,0,e_i\big) \ \big|  \ i\in [r] \Big\rangle\,,\]
as a $\ZZ$-module and so $\Ll=\ker_\ZZ A\cap \NN^n$ is generated as a semigroup by 
\begin{align*}
\bigg\{v_0+\sum_{j=1}^r v_j =\Big(\sum_{j=1}^re_j&,1,\sum_{j=1}^r e_j \Big),  ~7v_0+6\sum_{j=1}^r v_j=\Big(\sum_{j=1}^re_j,7,6\sum_{j=1}^r e_j \Big)\,, \\
 6v_0&+5\sum_{j=1}^r v_j=\Big(0,6,5\sum_{j=1}^r e_j \Big)\,, v_i=(6e_i,0,e_i) \ \Big|  \ i\in [r] \bigg\}\,.
\end{align*}
Indeed, an arbitrary element of $\Ll$ will be of the form
\[\alpha=a_0\Big(-5\sum_{j=1}^r e_j, 1, 0\Big)+\sum_{i=1}^r a_i(6e_i,0,e_i)=\Big(\sum_{j=1}^r(-5a_0+6a_j)e_j, a_0, \sum_{j=1}^r a_je_j \Big)\,,\]
where $a_i\gs 0$ for each $i=0,\ldots,r$ and $-5a_0+6a_j\gs 0$ for each $j=1,\ldots,r$ since all entries of $\alpha$ must be nonnegative. In particular, for each $i=1,\ldots, r$, we will have 

\begin{align*}
   a_i-6\lceil a_0/6\rceil+\lfloor a_o/6\rfloor\gs 0&,  \text{ if } 6|(a_0-1), \text{ and}\\
    a_i-4\lceil a_0/6\rceil-\lfloor a_o/6\rfloor\gs 0&,  \text{ otherwise.}
  \end{align*}
Then we can write
\begin{align*}\alpha=\left\lfloor \frac{a_0}{6}\right\rfloor\left(6v_0+5\sum_{j=1}^rv_j \right)+ \Big(\left\lceil \frac{a_0}{6}\right\rceil &- \left\lfloor \frac{a_0}{6}\right\rfloor\Big)\left(7v_0+ 6\sum_{j=1}^rv_j\right)\\
&+\sum_{i=1}^r\left(a_i-6\left\lceil \frac{a_0}{6}\right\rceil+\left\lfloor \frac{a_o}{6}\right\rfloor\right)v_i\,,
\end{align*}
if $6|(a_o-1)$, and otherwise
\begin{align*}\alpha=\left\lfloor \frac{a_0}{6}\right\rfloor\left(6v_0+5\sum_{j=1}^rv_j \right)+ \Big(\left\lceil \frac{a_0}{6}\right\rceil &- \left\lfloor \frac{a_0}{6}\right\rfloor\Big)\left(v_0+ \sum_{j=1}^rv_j\right)\\
&+\sum_{i=1}^r\left(a_i-4\left\lceil \frac{a_0}{6}\right\rceil-\left\lfloor \frac{a_o}{6}\right\rfloor\right)v_i\,,
\end{align*}
proving our claim. 

Let $\Ss\subseteq \Ll$ be the subsemigroup generated by all elements of $\Ll$ with support of size strictly less than $2r+1$. Our next claim is that this subsemigroup is generated by
\[\Big\{6v_0+5\sum_{j=1}^rv_j,~v_i \ \big| \ i\in[r] \Big\}\,.\]
Suppose $\gamma\in \Ll$ has support of size strictly less than $2r+1$, that is, it has at least one zero entry. As it belongs to $\Ll$ we can write
\[\gamma=g_1\Big(v_0+\sum_{j=1}^r v_j\Big)+g_2\Big(7v_0+6\sum_{j=1}^r v_j\Big)+g_3\Big(6v_0+5\sum_{j=1}^r v_j\Big)+\sum_{i=1}^r a_iv_i\,,\]
where $g_i,a_j$ are nonnegative integers. Hence, as $\gamma$ is equal to
\[(g_1+g_2+6a_1,\ldots,g_1+g_2+6a_r,g_1+7g_2+6g_3,g_1+6g_2+5g_3+a_1,\ldots, g_1+6g_2+5g_3+a_r)\,,\]
we must have $g_1=g_2=0$ since otherwise $\gamma$ has full support.

Our final claim is that $\Ss$ does not give a separating algebra. Indeed, the first $r+1$ entries of any element of $\Ss$ are divisible by 6, so for any prime $p$ and positive integer $m$, 
\[p^m\Big(v_0+\sum_{j=1}^r v_j\Big)=(p^m,\ldots,p^m, p^m,p^m,\ldots,p^m)\]
does not belong to $\Ss$. Hence by Proposition \ref{prop-sepalg}, $\Ss$ does not give a separating subalgebra.   \done
\end{eg}

%------------------------------------------------------------------------

\subsection{Minimal Size of monomial separating sets for Segre-Veroneses}

In this section, we study the minimal size of a monomial separating set for the affine cone over a Segre-Veronese variety. We consider the representation of a torus of rank $r$ whose ring of invariants is isomorphic to the ring of homogeneous coordinates on Segre-Veronese variety that is the image of the closed embedding $\prod_{i=1}^r \PP^{n_i-1} \hookrightarrow \PP^N$ given by the line bundle $\mcal{O}(a_1,\dots,a_r)$ as described in Section~\ref{section-SV}. Set $I:=\{i\mid a_i=1 \text{ or} \text{ a} \text{ pure} \text{ power} \text{ of}\allowbreak \chara \kk\}$. As in Section~\ref{section-SV}, the ring of invariants is given by:
\[S=\kk\big[\ x_0 M_1 \cdots M_r \ \big| \ M_i \, \text{is a monomial of degree $a_i$ in the variables} \  x_{i,1}, \dots , x_{i, n_i} \big]\,.\]

\begin{prop}\label{prop-SVminmon}
 \begin{enumerate}
  \item The monomial invariants with support of size at most $r+2$ form a separating set. \label{prop-SVminmon1}
  \item The minimal size of a monomial separating set is \[\Bigg(\prod_{h=1}^r n_h\Bigg) \Bigg(1 + \frac{1}{2}\sum_{i\notin I} ({n_i -1})\Bigg)\,.\] \label{prop-SVminmon2}
 \end{enumerate}
\end{prop}

\begin{proof}

As the torus is a reductive group by \cite[Theorem 2.3.16]{hd-gk:cit}, the restriction of any separating set to a subrepresentation must yield a separating set. As the restriction of an invariant monomial will be either zero or the same monomial, a monomial separating set must contain separating sets for any subrepresentation. 

Let $V$ be a $(r+1)$-dimensional subrepresentation of the action specified in the paragraph above. As any proper subset of the set of weights is linearly independent, a $(r+1)$-dimensional subrepresentation has no nonconstant invariants unless its matrix of weights is of the form
\[A:=\left( \begin{array}{c|c}
           I & \begin{array}{c} -a_1 \\ \vdots\\ -a_r\end{array} 
          \end{array}\right),\]
where $I$ is the $r\times r$ identity matrix. As $A$ is in row reduced echelon form over $\ZZ$, $\ker_\ZZ A$ is generated by $\alpha:=(a_1,\ldots,a_r,1)$ as a $\ZZ$-module. As $\alpha$ has positive entries, it must also generate $\ker_\ZZ A \cap \NN^{r+1}$ as a semigroup. That is, the ring of invariants is generated by exactly 1 monomial, and we get one such monomial for each of the 
$(r+1)$-dimensional subrepresentations with nonconstant invariants, of which there are $\prod_{i=1}^r {n_i}$.

Let $U$ be a $(r+2)$-dimensional subrepresentation of $W'$. By the same argument as before, its set of weights is the full set of weights, but one weight is repeated. For simplicity we suppose the repeated weight is $e_1$ so that the matrix of weights is
\[B:=\left( \begin{array}{c|cc}
           I & \begin{array}{cc} -a_1 &1 \\ -a_2 & 0  \\ \vdots & \vdots \\ -a_r & 0
          \end{array}\end{array}\right).\]
As $B$ is in reduced echelon form, its $\ZZ$-kernel is generated by ${\alpha_1:=(a_1,\ldots,a_r,1,0)}$ and ${\beta:=(-1,0,\ldots,0,0,1)}$ as a $\ZZ$-module. It follows that $\alpha_1$, 
\[{\alpha_2:=\alpha_1+a_1\beta=(0,a_2,\ldots,a_r,1,a_1)}\,,\] and \[\alpha_3:=\alpha_1+(a_1-1)\beta=(1,a_2,\ldots,a_r,1,a_1-1)\] also generate $\ker_\ZZ B$ as a $\ZZ$-module. 

We will use Lemma~\ref{lem-SSep} to show that the $\{\alpha_1,\alpha_2,\alpha_3\}$ give a separating set. To establish the first condition it suffices to note that $\alpha_1$ and $\alpha_2$ give a generator for the ring of invariants of the $(r+1)$-dimensional subrepresentation with support $\{1,\ldots,r+1\}$ and $\{2,\ldots,r+2\}$, respectively. To establish the second condition, we remark that $\alpha' \in \ker_\ZZ B \cap \NN^{r+2}$ will have support $\{1,\ldots,r+1\}$, $\{2,\ldots,r+2\}$ or $[r+2]$, and so we can take $\gamma$ equal to $\alpha_1$, $\alpha_2$, $\alpha_3$, respectively. Note that if $a_1=1$, then $\alpha_3=\alpha_1$ and so our $(r+2)$-dimensional representation does not require any new invariants beyond those needed for the $(r+1)$-dimensional subrepresentations. If $a_1=p^k$, where $p=\chara \kk$, then 
\begin{align*}p^k(1,a_2\ldots,a_r,1,a_1-1)&=a_1(1,a_2\ldots,a_r,1,a_1-1)\\
&=(a_1,\ldots,a_r,1,0)+(a_1-1)(0,a_2,\ldots,a_r,1,a_1)
\end{align*}
 belongs to the semigroup generated by $(a_1,\ldots,a_r,1,0)$ and $(0,a_2,\ldots,a_r,1,a_1)$, and so by Proposition~\ref{prop-SsepCharp} it follows that $(a_1,\ldots,a_r,1,0)$ and $(0,a_2,\ldots,a_r,1,a_1)$ give a separating set. Again, our $(r+2)$-dimensional representation does not require any new invariants. But if $a_1$ is not 1 or a pure power of $\chara\kk$, $\alpha_1$ and $\alpha_2$ do not give a separating set. Indeed, let $\zeta$ be a primitive $m$-th root of unity with $a_1/m$ a pure power of $\chara \kk$ and set $u_1=(1,1,\ldots,1)$ and $u_2=(\zeta,1,\ldots,1)$. Then $\alpha_1(u_1)=\alpha_1(u_2)=\alpha_2(u_1)=\alpha_2(u_1)=1$ but $\alpha_3(u_1)=1\neq\zeta=\alpha_3(u_2)$. Therefore we will require an extra monomial. Although this monomial need not be $\alpha_3$, it will have full support. Therefore  each $r+2$-dimensional subrepresentation with nonconstant invariants and repeated weight $i_0\notin I$ will necessitate at least one extra distinct monomial invariant. There are $\sum_{i_0 \notin I} \binom{n_i}{2} \prod_{i\neq i_0}{n_i}$ different such subrepresentations. It follows the minimal size of a separating set will be at least $\prod_{i=1}^rn_i+\sum_{i_0 \notin I} \binom{n_i}{2} \prod_{i\neq i_0}{n_i}$, which simplifies to the formula in Statement~(\ref{prop-SVminmon2}).

 We will use Lemma~\ref{lem-SSep} to show Statement~(\ref{prop-SVminmon1}). Denote by $\Ss$ the semigroup corresponding to the monomial algebra generated by all monomial invariants depending on at most $r+2$ variables and set $\Ll$ to be the semigroup giving all monomial invariants. The first step is to show that for any subrepresentation $V\subseteq W'$, $\Ll_{\supp(V)}\subseteq \ZZ \Ll_{\supp(V)}$. If $|\supp(V)|\ls r+2$, then the statement is trivially true. Suppose $|\supp(V)|> r+2$. By Corollary \ref{cor-r+1}, it follows that the rational invariants on $V$ are generated by the rational invariants with support of size at most $r+1$, and so in particular by the rational invariants with support of size at most $r+2$. Our argument of the previous paragraph shows that for any $(r+2)$-dimensional subrepresentation, the invariant monomials generate the field of rational invariants, therefore, the invariants monomials with support of size at most $r+2$ generated the rational invariants on $V$ as desired. 
 
 Now take $\alpha=(\alpha_0,\alpha_{1,1},\ldots,\alpha_{1,n_1},\ldots,\alpha_{r,1},\ldots,\alpha_{r,n_r})$ to be the exponent vector of a nonconstant invariant in $\kk[W']^{T'}$. As any proper subset of the set of distinct weights is linearly independent, $\wt(\alpha)=\{m_0,m_1,\ldots,m_r\}$. Hence, if $\alpha_{i_0,j_{i_0}}\neq 0$, we know that $\alpha_{0}\neq 0$ and for all $i\in [r]\setminus \{i_o\}$, there is $j_i\in[n_i]$ such that $\alpha_{i,j_i}\neq 0$. Define $\gamma\in\NN^{1+\sum_{i=1}^r n_i}$ as $\gamma_0=1$, $\gamma_{i_0,j_{i_0}}=a_{i_0}$, $\gamma_{i,j_i}=a_i$ for all other $i\in[n]$, with the remaining entries zero. Then
 \[\gamma_0m_0+\sum_{i=1}^r\sum_{j_i=1}^{n_i}\gamma_{i,j_i}m_{i,j_i}=-\sum_{i=1}^ra_ie_i+\sum_{i=1}^ra_ie_i=0,\]
 and so $\gamma$ gives an invariant. By construction $\gamma_{i_0,j_{i_0}}\neq 0$ and $\gamma$ has support of size $r+2$ contained in the support of $\alpha$. We have now established that $\Ss$ satisfies the two conditions of Lemma \ref{lem-SSep} and so the monomial invariants with support of size at most $r+2$ form a separating set, proving Statement~(\ref{prop-SVminmon1}).
 
A consequence of Statement~(\ref{prop-SVminmon1}) is that to have a monomial separating set for the full representation is suffices to have a set of monomials that restricts to a separating set for each $(r+2)$-dimensional subrepresentation. Our argument in the first two paragraphs of the proof gives a construction for such a separating set which has size equal to the formula in Statement~(\ref{prop-SVminmon2}), completing the proof.
\end{proof}

\begin{cor}\label{cor-sparse} The elements in a separating set must contain at least 
\[\Bigg(\prod_{h=1}^r n_h\Bigg) \Bigg(1 + \frac{1}{2}\sum_{i\notin I} ({n_i -1})\Bigg)\,\]
monomials between them altogether.
\end{cor}
\begin{proof} As the ring of invariants is generated by monomials, the set of monomials contained in the elements of any separating set must form a monomial separating set. The previous proposition applies to this set.
\end{proof}

\section*{Acknowledgements}

The authors would like to thank Anurag Singh and Bernd Sturmfels for fruitful suggestions and questions on the content of this work. We also thank Mateusz Micha\l{}ek and Emre Sert\"oz for pointing out an error in an earlier version, as well as the anonymous referee for catching mistakes and improving the presentation of this manuscript. The first author acknowledges support from the LMS through a Grace Chisholm Young fellowship. The second author acknowledges support by the NSF grant DMS-0943832.

%%%%%%%%%%%%%%%%%%%%%%%%%%%%%%%%%%%%%%%%%%%%%%%%%%%%%%%%%%%%%%
%%%%%%%%%%%%%%%%%%%%%%%%%%%%%%%%%%%%%%%%%%%%%%%%%%%%%%%%%%%%%
%%%%%%%%%%%%%%%%%%%%%%%%%%%%%%%%%%%%%%%%%%%%%%%%%%%%%%%%%%%%%
%%%%%%%%%%%%%%%%%%%%%%%%%%%%%%%%%%%%%%%%%%%%%%%%%%%%%%%%%%%%%

\bibliographystyle{plain}
\bibliography{references_repn_of_tori}

\end{document}